\documentclass[journal]{IEEEtran}
\usepackage[T1]{fontenc}
\usepackage[latin9]{inputenc}
\usepackage{amsmath}
\usepackage{amsthm}
\usepackage{amssymb}
\usepackage{algorithm,algorithmic}
\usepackage{graphicx,color}

\newcounter{thm}
\newtheorem{lemma}[thm]{\bf Lemma}

\newtheorem{fact}[thm]{\bf Fact}
\newtheorem{theorem}[thm]{\bf Theorem}
\newtheorem{corollary}[thm]{\bf Corollary}

\newtheorem{definition}{Definition}
\newtheorem{assumption}{Assumption}

\makeatother

\ifCLASSINFOpdf
\else
\fi

\begin{document}

\title{A {Universal} Empirical Dynamic Programming Algorithm for Continuous {State} MDPs}


\author{William B. Haskell, Rahul Jain, Hiteshi Sharma and Pengqian Yu 
\thanks{W.B. Haskell and P. Yu are with the Department of Industrial and Systems
Engineering, National University of Singapore.}
\thanks{Rahul Jain, and Hiteshi Sharma are with EE Department, University
of Southern California. The second and third authors' work was supported by an ONR Young Investigator Award \#N000141210766 and by NSF Award  CCF-1817212. {A preliminary version of this paper appeared in CDC 2017 \cite{haskell2017randomized}.}}
\thanks{Manuscript submitted: August 20, 2017. Revised: \today} }

\maketitle

\begin{abstract}
We propose universal randomized function approximation-based
empirical value learning (EVL) algorithms for Markov decision processes.
The `empirical' nature comes from each iteration being done empirically
from samples available from simulations of the next state. This makes
the Bellman operator a random operator. A parametric and a non-parametric
method for function approximation using a parametric function space
and a Reproducing Kernal Hilbert Space (RKHS) respectively are then
combined with EVL. Both function spaces have the universal function
approximation property. Basis functions are picked randomly. Convergence
analysis is done using a random operator framework with techniques
from the theory of stochastic dominance. Finite time sample complexity
bounds are derived for both universal approximate dynamic programming
algorithms. Numerical experiments support the versatility and {computational tractability}
of this approach. 
\end{abstract}



\section{Introduction}

\label{sec:intro}

There exist a wide variety of approximate dynamic programming (DP)
\cite[Chapter 6]{bertsekas2011dynamic}, \cite{powell2007approximate}
and reinforcement learning (RL) algorithms \cite{sutton1998reinforcement}
for finite state space Markov decision processes (MDPs). But many
real-world problems of interest have either a continuous state space,
or very large state space that it is best approximated as one. Action space will be considered finite.
Approximate DP and RL algorithms do exist for continuous state space MDPs but choosing
which one to employ is an art form: different techniques (state space
aggregation and function approximation \cite{devore1998nonlinear})
and algorithms work for different problems {\cite{Bertsekas_ADP_2005,Bertsekas_ADP_2010,Powell_ADP_2007}}, and universally applicable
algorithms are lacking. For example, fitted value iteration \cite{munos2008finite}
is very effective for some problems but requires the choice of an
appropriate basis functions for good approximation. 
{Most of the existing work on approximate dynamic programming (ADP) requires domain knowledge of the problem at hand for effective implementation. Here, we are interested in ADP methods which are effective without any prior problem knowledge.}

In this paper, we propose approximate DP algorithms for continuous
state space MDPs with finite action space that are universal (approximating function space
can provide arbitrarily good approximation for any problem), computationally
tractable, simple to implement and yet we have non-asymptotic sample
complexity bounds. The first is accomplished by picking functions
spaces for approximation that are dense in the space of continuous
functions {(i.e., for any continuous function $f$, and $\epsilon>0$, there is an element of
our approximating function space that is within $\epsilon$ of $f$ in the sup-norm.)}
The second goal is achieved by relying on randomized selection
of basis functions for approximation and also by `empirical' dynamic
programming \cite{haskell2016empirical}. The third is enabled because
standard Python routines can be used for function fitting and the
fourth is by analysis in a random operator framework which provides
non-asymptotic rate of convergence and sample complexity bounds.


There is a large body of well-known literature on reinforcement learning
and approximate dyamic programming for continuous state space MDPs.
We discuss the most directly related. In \cite{rust1997using}, a
sampling-based state space aggregation scheme combined with sample
average approximation for the expectation in the Bellman operator
is proposed. Under some regularity assumptions, the approximate value
function can be computed at any state and an estimate of the expected
error is given. But the algorithm seems to suffer from poor numerical
performance. A linear programming-based constraint-sampling approach
was introduced in \cite{de2004constraint}. Finite sample error guarantees,
with respect to this constraint-sampling distribution are provided
but the method suffers from issues of feasibility. The closest paper
to ours is \cite{munos2008finite} that does function fitting with
a given basis and does `empirical' value iteration in each step. Unfortunately,
it is not a universal method as approximation quality depends on the
function basis picked. Other papers worth noting are \cite{ormoneit2002kernel}
that discusses kernel-based value iteration and the bias-variance
tradeoff, and \cite{grunewalder2012modelling} that proposed a kernel-based
algorithm with random sampling of the state and action spaces, and
proves asymptotic convergence. {Other related works worth mentioning are \cite{szepesvari2001efficient,munos2007performance} (approximate value iteration), \cite{de2003linear,bhat2012non} (the LP approach to approximate DP) and \cite{munos2003error,konda2000actor,jain2010simulation} (approximate policy iteration). Recent applications stress policy gradient methods \cite{sutton2000policy,peters2016policy} and deep learning-based function approximation \cite{mnih2015human} for which theoretical performance guarantees for general problems are not available. The method presented in this paper may be seen as another alternative. }

This paper is inspired by the `random function' approach that uses
randomization to (nearly) solve otherwise intractable problems (see
e.g., \cite{rahimi2008uniform,rahimi2009weighted}) and the `empirical'
approach that reduces computational complexity of working with expectations
\cite{haskell2016empirical}. We propose two new algorithms. For the first
parametric approach, we pick a parametric function family. In each
iteration a number of functions are picked randomly for function fitting
by sampling the parameters. {A preliminary version of this for $l_2$ function fitting appeared in \cite{haskell2017randomized}. }
For the second non-parametric approach,
we pick a RKHS for approximation. Both function spaces are dense in
the space of continuous functions. In each iteration, we sample a
few states from the state space. Empirical value learning (EVL) is
then performed on these states. Each step of EVL involves approximating
the Bellman operator with an empirical (random) Bellman operator by
plugging a sample average approximation from simulation for the expectation.
This is akin to doing stochastic approximations with step size one.
{We employ a probabilistic convergence analysis technique of iterated random operators based on stochastic dominance that we developed in \cite{haskell2016empirical}. This method is general in the sense that not only can we handle various norms, but also various random contractive operators.}

The main contribution of this paper is development of randomized function
approximation-based (offline) dynamic programming algorithms that
are universally applicable (i.e., do not require appropriate choice
of basis functions for good approximation). A secondary contribution
is further development of the random operator framework for convergence
analysis in the $\mathcal{L}_{p}-$norm that also yields finite time
sample complexity bounds.

The paper is organized as follows. Section \ref{sec:prelim} presents preliminaries including the continuous state space MDP model and the empirical dynamic programming framework for finite state MDPs introduced in \cite{haskell2016empirical}. Section \ref{sec:algos} presents two empirical value learning algorithms - first, a randomized parametric function fitting method, and second, a non-parametric randomized function fitting in an RKHS space. We also provide statements of main theorems about non-asymptotic error guarantees. Section \ref{sec:analysis} presents a unified analysis of the two algorithms in a random operator framework. Numerical results are reported in Section \ref{sec:numerical}. Supplemental proofs are relegated to the appendix.

\section{Preliminaries}\label{sec:prelim}

Consider a discrete time discounted MDP given by the 5-tuple, $\left(\mathbb{S},\,\mathbb{A},\,Q,\,c,\,\gamma\right)$.
The state space $\mathbb{S}$ is a compact subset of $\mathbb{R}^{d}$
with the Euclidean norm, with corresponding Borel $\sigma-$algebra
$\mathcal{B}\left(\mathbb{S}\right)$. Let $\mathcal{F}\left(\mathbb{S}\right)$
be the space of all $\mathcal{B}\left(\mathbb{S}\right)-$measurable
bounded functions $f\text{ : }\mathbb{S}\rightarrow\mathbb{R}$ in
the supremum norm $\|f\|_{\infty}:=\sup_{s\in\mathbb{S}}|f\left(s\right)|$.
Moreover, let $\mathcal{M}\left(\mathbb{S}\right)$ be the space of
all probability distributions over $\mathbb{S}$ and define the $\mathcal{L}_{p}$
norm as {$\|f\|_{p,\,\mu}^{p}:=\left(\int_{\mathbb{S}}|f\left(s\right)|^{p}\mu\left(ds\right)\right)$
for $p\in[1,\,\infty)$} and given $\mu\in\mathcal{M}\left(\mathbb{S}\right)$. 
We assume that the action space $\mathbb{A}$ is finite. The transition law $Q$ governs the system evolution. For $B\in\mathcal{B}\left(\mathbb{S}\right)$,
$Q\left(B\,\vert\,s,\,a\right)$ is the probability of next visiting
the set $B$ given that action $a\in\mathbb{A}$ is chosen in state
$s\in\mathbb{S}$. The cost function $c\mbox{ : }\mathbb{S}\times\mathbb{A}\rightarrow\mathbb{R}$
is a bounded measurable function that depends on state-action pairs. Finally,
$\gamma\in\left(0,\,1\right)$ is the discount factor.

We will denote by $\Pi$ the class of \textit{stationary deterministic
Markov policies}: mappings $\pi\mbox{ : }\mathbb{S}\rightarrow\mathbb{A}$
which only depend on history through the current state. For a given
state $s\in\mathbb{S}$, $\pi\left(s\right)\in\mathbb{A}$ is the
action chosen in state $s$ under the policy $\pi$. The state and
action at time $t$ are denoted $s_{t}$ and $a_{t}$, respectively.
Any policy $\pi\in\Pi$ and initial state $s\in\mathbb{S}$ determine
a probability measure $P_{s}^{\pi}$ and a stochastic process $\left\{ \left(s_{t},\,a_{t}\right),\,t\geq0\right\} $
defined on the canonical measurable space of trajectories of state-action
pairs. The expectation operator with respect to $P_{s}^{\pi}$ is
denoted $\mathbb{E}_{s}^{\pi}\left[\cdot\right]$.

We will assume that the cost function $c$ satisfies $|c\left(s,\,a\right)|\leq c_{\max}<\infty$
for all $\left(s,\,a\right)\in\mathbb{S}\times\mathbb{A}$. Under
this assumption, $\|v^{\pi}\|_{\infty}\leq v_{\max}:=c_{\max}/\left(1-\gamma\right)$
where $v^{\pi}$ is the value function for policy $\pi\in\Pi$ defined
as $v^{\pi}\left(s\right)=\mathbb{E}_{s}^{\pi}\left[\sum_{t=0}^{\infty}\gamma^{t}c\left(s_{t},\,a_{t}\right)\right],\,\forall s\in\mathbb{S}$.
For later use, we define $\mathcal{F}\left(\mathbb{S};\,v_{\max}\right)$
to be the space of all functions $f\in\mathcal{F}\left(\mathbb{S}\right)$
such that $\|f\|_{\infty}\leq v_{\max}$.

The optimal value function is $v^{*}\left(s\right):=\inf_{\pi\in\Pi}\mathbb{E}_{s}^{\pi}\left[\sum_{t=0}^{\infty}\gamma^{t}c\left(s_{t},\,a_{t}\right)\right],\,\forall s\in\mathbb{S}.$
To characterize the optimal value function, we define the Bellman
operator $T\mbox{ : }\mathcal{F}\left(\mathbb{S}\right)\rightarrow\mathcal{F}\left(\mathbb{S}\right)$
via 
\[
\left[T\,v\right]\left(s\right):=\min_{a\in\mathbb{A}}\left\{ c\left(s,\,a\right)+\gamma\,\mathbb{E}_{X\sim Q\left(\cdot\,\vert\,s,\,a\right)}\left[v\left(X\right)\right]\right\} ,\,\forall s\in\mathbb{S}.
\]
It is well known that the optimal value function $v^{*}$ is a fixed point
of $T$, i.e. $T\,v^{*}=v^{*}$ {\cite[Theorem 6.2.5]{puterman2014markov}.} Classical value iteration is based
on iterating $T$ to obtain a fixed point, it produces a sequence
$(v_{k})_{k\geq0}\subset\mathcal{F}\left(\mathbb{S}\right)$ given
by $v_{k+1}=T\,v_{k},\,k\geq0.$ Also, we know that $(v_{k})_{k\geq0}$
converges to $v^{*}$ geometrically in $\|\cdot\|_{\infty}$.

We are interested in approximating the optimal value function $v^{*}$
within a tractable class of approximating functions $\mathcal{F}\subset\mathcal{F}\left(\mathbb{S}\right)$.
We have the following definitions which we use to measure the approximation
power of $\mathcal{F}$ with respect to $T$. {We define 
\[d_{p,\,\mu}\left(\mathcal{G},\mathcal{F}\right):=\sup_{g\in\mathcal{G}}\inf_{f\in\mathcal{F}}\|f-g\|_{p,\,\mu}\]
to be the distance between two function classes; then
$d_{p,\,\mu}\left(T\,\mathcal{F},\,\mathcal{F}\right)$
is the inherent $\mathcal{L}_{p}$ \textit{Bellman error for the
function class} $\mathcal{F}$. Similarly, defining
\[
d_{\infty}\left(\mathcal{G},\mathcal{F}\right):=\sup_{g\in\mathcal{G}}\inf_{f\in\mathcal{F}}\|f-g\|_{\infty}
\]
gives $d_{\infty}\left(T\mathcal{F},\mathcal{F}\right)$ as the inherent $\mathcal{L}_{\infty}$
Bellman error for an approximating class $\mathcal{F}$.}

We often compare $\mathcal{F}$ to the Lipschitz continuous functions
$\text{Lip}\left(L\right)$ defined as 
\[
\left\{ f\in\mathcal{F}\left(\mathbb{S}\right)\text{ : }|f\left(s\right)-f\left(s'\right)|\leq L\,\|s-s'\|,\,\forall s,\,s'\in\mathbb{S}\right\} .
\]
{In our case, we say that an approximation class $\mathcal{F}$ is
\textit{universal} if $d_{\infty}\left(\text{Lip}\left(L\right),\,\mathcal{F}\right)=0$
for all $L\geq0$. Note that on a compact state space $\mathbb{S}$,
universality in the supremum norm implies universality in the $\mathcal{L}_{1}$
and $\mathcal{L}_{2}$ norms as well.} 

One of the difficulties of dynamic programming algorithms like value
iteration above is that each iteration of the Bellman operator involves
computation of an expectation which may be expensive. Thus, in \cite{haskell2016empirical},
we proposed replacing the Bellman operator with an empirical (or random)
Bellman operator, 
\[
\left[\hat{T}_{n}v\right]\left(s\right):=\min_{a\in\mathbb{A}}\left\{ c\left(s,\,a\right)+\frac{\gamma}{n}\sum_{i=1}^{n}[v(X_{i})]\right\} ,
\]
where $X_{i}$ are samples of the next state from $Q\left(\cdot\,\vert\,s,\,a\right)$
which can be obtained from simulation. Now, we can iterate the empirical
Bellman operator, 
\[
v_{k+1}=\hat{T}_{n}v_{k}, ~~~\forall k \geq 0,
\]
an algorithm we called Empirical Value Iteration (EVI). The sequence
of iterates $\{{v}_{k}\}$ is a random process. Since $T$ is
a contractive operator, its iterates converge to its fixed point
$v^{*}$. The random operator $\hat{T}_{n}$ may be expected to inherit
the contractive property in a probabilistic sense and its iterates
converge to some sort of a probabilitic fixed point. We introduce
$(\epsilon,\delta)$ versions of two such notions introduced in \cite{haskell2016empirical}.

\begin{definition} A function $\hat{v}:\mathbb{S}\to\mathbb{R}$
is an $(\epsilon,\delta)$-strong probabilistic fixed point for a
sequence of random operators $\{\hat{T}_{n}\}$ if there exists an
$N$ such that for all $n>N$, 
\[
\mathbb{P}\left(||\hat{T}_{n}\hat{v}-\hat{v}||>\epsilon\right)<\delta.
\]
\end{definition} It is called a \textit{strong probabilistic fixed
point}, if the above is true for every positive $\epsilon$ and $\delta$.

\begin{definition} A function $\hat{v}:\mathbb{S}\to\mathbb{R}$
is an $(\epsilon,\delta)$-weak probabilistic fixed point for a sequence
of random operators $\{\hat{T}_{n}\}$ if there exist $N$ and $K$
such that for all $n>N$ and all $k>K$, 
\[
\mathbb{P}\left(||\hat{T}_{n}^{k}v_0-\hat{v}||>\epsilon\right)<\delta,~~~\forall v_0\in\mathcal{F}\left(\mathbb{S}\right).
\]
\end{definition} It is called a \textit{weak probabilistic fixed
point}, if the above is true for every positive $\epsilon$ and $\delta$.
Note that the stochastic iterative algorithms such as EVL often find
the weak probabilistic fixed point of $\{\hat{T}_{n}\}$ whereas what
we are looking for is $v^{*}$, the fixed point of $T$. In \cite{haskell2016empirical},
it was shown that asymptotically the weak probabilistic fixed point
of $\{\hat{T}_{n}\}$ coincides with its strong probabilistic fixed
points which coincide with the fixed point of $T$ under certain fairly
weak assumptions and a natural relationship between $T$ and $\{\hat{T}_{n}\}$
\[
\lim_{n\to\infty}\mathbb{P}\left(||\hat{T}_{n}{v}-T{v}||>\epsilon\right)=0,~~~\forall v\in\mathcal{F}\left(\mathbb{S}\right).
\]
This implies that stochastic iterative algorithms such as EVL will
find approximate fixed points of $T$ with high probability.



\section{The Algorithms and Main Results}

\label{sec:algos}

When the state space $\mathbb{S}$ is very large, or even uncountable,
exact dynamic programming methods are not practical,  or even feasible.
Instead, one must use a variety of approximation methods. In particular,
function approximation (or fitting the value function with a fixed
function basis) is a common technique. The idea is to sample a finite
set of states from $\mathbb{S}$, approximate the Bellman update at
these states, and then extend to the rest of $\mathbb{S}$ through
function fitting similar to \cite{munos2008finite}. Furthermore,
the expectation in the Bellman operator, for example, is also approximated
by taking a number of samples of the next state. There are two main
difficulties with this approach: First, the function fitting depends
on the function basis chosen, making the results problem-dependent.
Second, with a large basis (for good approximation), function fitting
can be computationally expensive.

In this paper, we aim to address these issues by first picking universal
approximating function spaces, and then using randomization to pick
a smaller basis and thus reduce the computational burden of the function
fitting step. We consider two functional families, one is a parametric
family $\mathcal{F}(\Theta)$ parameterized over parameter space $\Theta$
and the other is a non-parametric regularized RKHS. By $\mu\in\mathcal{M}\left(\mathbb{S}\right)$,
we will denote a probability distribution from which to sample states
in $\mathbb{S}$, and by a $\mathcal{F}\subset\mathcal{F}\left(\mathbb{S};\,v_{\max}\right)$
we will denote a functional family in which to do value function approximation.

Let us denote by $(v_{k} )_{k\geq0}\subset\mathcal{F}\left(\mathbb{S};\,v_{\max}\right)$,
the iterates of the value functions produced by an algorithm and a
sample of size $N\geq1$ from $\mathbb{S}$ is denoted $s^{1:N}=\left(s_{1},\ldots,\,s_{N}\right)$.
The empirical $p-$norm of $f$ is defined as $\|f\|_{p,\,\hat{\mu}}^{p}:=\frac{1}{N}\sum_{n=1}^{N}|f\left(s_{n}\right)|^{p}$
for $p\in[1,\,\infty)$ and as $\|f\|_{\infty,\,\hat{\mu}}:=\sup_{n=1,\ldots,\,N}|f\left(s_{n}\right)|$
for $p=\infty$, where $\hat{\mu}$ is the empirical measure coresponding
to the samples $s^{1:N}$.


We will make the following technical assumptions for the rest of the
paper similar to those made in \cite{munos2008finite}. 
\begin{assumption}
\label{assu:Kernel} (i) For all $\left(s,\,a\right)\in\mathbb{S}\times\mathbb{A}$,
$Q\left(\cdot\,\vert\,s,\,a\right)$ is absolutely continuous with
respect to $\mu$ and 
\[
C_{\mu}:=\sup_{\left(s,\,a\right)\in\mathbb{S}\times\mathbb{A}}\left\Vert dQ\left(\cdot\,\vert\,s,\,a\right)/d\mu\right\Vert _{\infty}<\infty.
\]
(ii) Given any sequence of policies $\left\{ \pi_{m}\right\} _{m\geq1}$,
the future state distribution $\rho\,Q^{\pi_{1}}\cdots Q^{\pi_{m}}$
is absolutely continuous with respect to $\mu$, 
\[
c_{\rho,\,\mu}\left(m\right):=\sup_{\pi_{1},\ldots,\,\pi_{m}}\left\Vert d\left(\rho\,Q^{\pi_{1}}\cdots Q^{\pi_{m}}\right)/d\mu\right\Vert _{\infty}<\infty,
\]
and $C_{\rho,\,\mu}:=\sum_{m\geq0}\gamma^{m}c_{\rho,\,\mu}\left(m\right)<\infty$.
\end{assumption} 
The above assumptions are conditions on transition probabilities, the first being a sufficient condition for the second. {$\rho$ can be regarded as an ``importance'' distribution on $\mathbb{S}$, that is possibly different from the distribution $\mu$ on $\mathbb{S}$ that is used to sample states. Assumption 1 is essentially a regularity condition on the MDP: It ensures that the MDP cannot make arbitrary transitions with high probability with respect to the initial state distribution $\mu$.} {$C_{\rho,\mu}$ is called the discounted-average concentrability
coefficient of the future-state distributions in \cite{munos2008finite}. Note that the assumption is satisfied when $\mu$ is the Lebesgue measure
on $\mathbb{S}$ and the transition kernel has a bounded density with respect to $\mu$.}

\subsection{Random Parametric Basis Function (RPBF) Approximation}
\label{sec:rpbf}

We introduce an empirical value learning algorithm with function
approximation using random parametrized basis functions (EVL+RPBF).
It requires a parametric family $\mathcal{F}$ built from a set of
parameters $\Theta$ with probability distribution $\nu$ and a feature
function $\phi\text{ : }\mathbb{S}\times\Theta\rightarrow\mathbb{R}$
(that depends on both states and parameters) with the assumption that
$\sup_{\left(s,\,\theta\right)\in\mathbb{S}\times\Theta}|\phi\left(s;\,\theta\right)|\leq1$.
This can easily be met in practice by scaling $\phi$ whenever $\mathbb{S}$
and $\Theta$ are both compact and $\phi$ is continuous in $\left(s,\,\theta\right)$.
Let $\alpha\text{ : }\Theta\rightarrow\mathbb{R}$ be a weight function
and define $\mathcal{F}\left(\Theta\right):=$ 
\[
\left\{ f\left(\cdot\right)=\int_{\Theta}\phi\left(\cdot;\,\theta\right)\alpha\left(\theta\right)d\theta\,:\,|\alpha\left(\theta\right)|\leq C\,\nu\left(\theta\right),\,\forall\theta\in\Theta\right\} .
\]
We note that the condition $|\alpha\left(\theta\right)|\leq C\,\nu\left(\theta\right)$
for all $\theta\in\Theta$ is equivalent to requiring that $\|\alpha\|_{\infty,\,\nu}:=\sup_{\theta\in\Theta}|\alpha\left(\theta\right)/\nu\left(\theta\right)|\leq C$
where $\|\alpha\|_{\infty,\,\nu}$ is the $\nu-$weighted supremum
norm of $\alpha$ and $C$ is a constant.

The function space $\mathcal{F}\left(\Theta\right)$ may be chosen
to have the `universal' function approximation property in the sense
that any Lipschitz continuous function can be approximated arbitrarily
closely in this space as shown in \cite{rahimi2008uniform}. By \cite[Theorem 2]{rahimi2008uniform},
many such choices of $\mathcal{F}\left(\Theta\right)$ are possible
and are developed in \cite[Section 5]{rahimi2008uniform}. For example,
$\mathcal{F}\left(\Theta\right)$ is universal in the following two
cases: 

\begin{itemize}
\item $\phi\left(s;\,\theta\right)=\cos\left(\langle\omega,\,s\rangle+b\right)$
where $\theta=\left(\omega,\,b\right)\in\mathbb{R}^{d+1}$; and $\nu\left(\theta\right)$
is given by $\omega\sim\text{Normal}\left(0,\,2\,\gamma\,I\right)$
and $b\sim\text{Uniform}\left[-\pi,\,\pi\right]$; 
\item $\phi\left(s;\,\theta\right)=\text{sign}\left(s_{k}-t\right)$ where
$\theta=\left(t,\,k\right)\in\mathbb{R}\times\left\{ 1,\ldots,\,d\right\} $;
and $\nu\left(\theta\right)$ to be given by $k\sim\text{Uniform}\left\{ 1,\ldots,\,d\right\} $
and $t\sim\text{Uniform}\left[-a,\,a\right]$. 
\end{itemize}

In this approach, we have a parametric function family $\mathcal{F}\left(\Theta\right)$
but instead of optimizing over parameters in $\Theta$, we randomly
sample them first and then do function fitting which involves optimizing
over finite weighted combinations $\sum_{j=1}^{J}\alpha_{j}\phi\left(\cdot;\,\theta_{j}\right)$.
Unfortunately, this leads to a non-convex optimization problem. Hence,
instead of optimizing over $\theta^{1:J}=\left(\theta_{1},\ldots,\,\theta_{J}\right)$
and $\alpha^{1:J}=\left(\alpha_{1},\ldots,\,\alpha_{J}\right)$ jointly,
we first do randomization over $\theta^{1:J}$ and then optimization
over $\alpha^{1:J}$, as in \cite{rahimi2009weighted}, to bypass
the non-convexity inherent in optimizing over $\theta^{1:J}$ and
$\alpha^{1:J}$ simultaneously.
This approach allows us to deploy rich parametric families without
much additional computational cost. Once we draw a random sample $\left\{ \theta_{j}\right\} _{j=1}^{J}$
from $\Theta$ according to $\nu$, we obtain a random function space:
$\widehat{\mathcal{F}}\left(\theta^{1:J}\right):=$
\vspace{-2.2cm} 
\[
\left\{ f\left(\cdot\right)=\sum_{j=1}^{J}\alpha_{j}\phi\left(\cdot;\,\theta_{j}\right)\,:\,\|\left(\alpha_{1},\ldots,\,\alpha_{J}\right)\|_{\infty}\leq C/J\right\}.
\vspace{-0.5cm}
\]

Step 1 of such an algorithm (Algorithm \ref{algo:EVL+RPBF}) involves
sampling states $s^{1:N}$ over which to do value iteration and sampling
parameters $\theta^{1:J}$ to pick basis functions $\phi(\cdot;\theta)$
which are used to do function fitting. Step 2 involves doing an empirical
value iteration over states $s^{1:N}$ by sampling next states $(X_{m}^{s_{n},\,a})_{m=1}^{M}$
according to the transition kernel $Q$, and using the current iterate
of the value function $v_{k}$. Note that fresh (i.i.d.) samples of
the next state are regenerated in each iteration. Step 3 involves
finding the best fit to $\tilde{v}_{k}$, the iterate from Step 2,
within $\widehat{\mathcal{F}}\left(\theta^{1:J}\right)$ wherein randomly
sampled parameters $\theta^{1:J}$ specify the basis functions for
function fitting and weights $\alpha^{1:J}$ are optimized, which
is a convex optimization problem.

\begin{algorithm}[H]
\caption{\label{algo:EVL+RPBF} EVL with random parameterized basis functions
(EVL+RPBF)}
Input: probability distribution $\mu$ on $\mathbb{S}$ and $\nu$
on $\Theta$;\\
 Sample sizes $N\geq1,$ $M\geq1$, $J\geq1$; initial seed $v_{0}$.
{counter $k=0$ and  iterations $K\geq1$.}\\
 For $k=1,\ldots,K$
\begin{enumerate}
\item Sample $(s_{n})_{n=1}^{N}\sim\mu^{N}$ and $(\theta_{j})_{j=1}^{J}\sim\nu^{J}$. 
\item Compute 
\[
\tilde{v}_{k}\left(s_{n}\right)=\min_{a\in\mathbb{A}}\left\{ c\left(s_{n},\,a\right)+\frac{\gamma}{M}\sum_{m=1}^{M}{v}_{k}\left(X_{m}^{s_{n},\,a}\right)\right\} ,
\]
where $(X_{m}^{s_{n},\,a})\sim Q\left(\cdot\,\vert\,s_{n},\,a\right)$,
$m=1,\cdots,M$ are i.i.d. 
\item $\alpha^{k}=\arg\min_{\alpha}\,\frac{1}{N}\sum_{n=1}^{N}\left(\sum_{j=1}^{J}\alpha_{j}\phi\left(s_{n};\,\theta_{j}\right)-\tilde{v}\left(s_{n}\right)\right)^{2}\\\text{s.t.}\,\|\left(\alpha_{1},\ldots,\,\alpha_{J}\right)\|_{\infty}\leq C/J$.\\
 ${v}_{k+1}(s)=\sum_{j=1}^{J}\alpha_{j}^{k}\phi(s;\theta_{j})$. 
\item Increment $k\leftarrow k+1$ and return to Step 1. 
\end{enumerate}
\end{algorithm}

We note that Step 3 of the algorithm can be replaced by another method
for function fitting (as we do in the next subsection). The above
algorithm differs from Fitted Value Iteration (FVI) algorithm of \cite{munos2008finite}
in how it does function fitting. FVI does function fitting with a
deterministic and given set of basis functions which limits its universality
while we do function fitting in a much larger space which has the universal
function approximation property but are able to reduce computational
complexity by exploiting randomization.

In \cite[Section 7]{munos2008finite}, it is shown that if the transition
kernel and cost are smooth {such that there exist $L_Q$ and $L_c$ for which}
\begin{equation}
\|Q\left(\cdot\,\vert\,s,\,a\right)-Q\left(\cdot\,\vert\,s',\,a\right)\|_{TV}\leq L_{Q}\|s-s'\|_{2}\label{eq:smooth}
\end{equation}
and 
\begin{equation}
|c\left(s,\,a\right)-c\left(s',\,a\right)|\leq L_{c}\|s-s'\|_{2}\label{eq:smooth-1}
\end{equation}
hold for all $s,\,s'\in\mathbb{S}$ and $a\in\mathbb{A}$, then the
Bellman operator $T$ maps bounded functions to Lipschitz continuous
functions. In particular, if $v$ is uniformly bounded by $v_{\max}$
then $T\,v$ is $\left(L_{c}+\gamma\,v_{\max}L_{Q}\right)-$Lipschitz
continuous. Subsequently, the inherent $\mathcal{L}_{\infty}$ Bellman
error satisfies $d_{\infty}\left(T\,\mathcal{F},\,\mathcal{F}\right)\leq d_{\infty}\left(\text{Lip}\left(L\right),\,\mathcal{F}\right)$
since $T\,\mathcal{F}\subset\text{Lip}\left(L\right)$. So, it only
remains to choose an $\mathcal{F}\left(\Theta\right)$ that is dense
in $\text{Lip}\left(L\right)$ in the supremum norm, for which many
examples exist.


We now provide non-asymptotic sample complexity bounds to establish
that Algorithm \ref{algo:EVL+RPBF} yields an approximately optimal
value function with high probability. We provide guarantees for both
the $\mathcal{L}_{1}$ and $\mathcal{L}_{2}$ metrics on the error.

Denote 
\begin{eqnarray*}
	N_{2}(\varepsilon,\delta') & = & 2^{7}5^{2}\bar{v}_{\max}^{4}\log\left[\frac{40\,e\left(J_2+1\right)}{\delta}\left(10\,e\,\bar{v}_{\max}^{2}\right)^{J}\right],\\
	M_{2}(\varepsilon,\delta') & = & \left(\frac{\bar{v}_{\max}^{2}}{2}\right)\log\left[\frac{10\,N_2\,|\mathbb{A}|}{\delta'}\right],\\
	J_{2}(\varepsilon,\delta') & = & \left(\cfrac{5C}{\varepsilon}\left(1+\sqrt{2\log\cfrac{5}{\delta'}}\right)\right)^{2},~~~\text{and}\\
	K_{2}^{*} & = & 2\left\lceil \frac{\ln\left(C_{\rho,\,\mu}^{1/2}\varepsilon\right)-\ln\left(2\,v_{\max}\right)}{\ln\,\gamma}\right\rceil ,
\end{eqnarray*}
where  $\bar{v}_{max}=v_{\max}/\varepsilon$. Set $\delta':=1-\left(1-\delta/2\right)^{1/\left(K_{2}^{*}-1\right)}$.
Then, we have the following sample complexity bound on Algorithm \ref{algo:EVL+RPBF} with $\mathcal{L}_{2}$ error. We note that $\mathcal{L}_{2,\,\mu}\left(\mathbb{S}\right)$ is a Hilbert space and that many powerful function approximation results exist for this setting because of the favorable properties of a Hilbert space.

\begin{theorem} \label{thm:RPBF_L2} Given an $\varepsilon>0$, and
	a $\delta\in\left(0,\,1\right)$, choose $J\geq J_{2}(\varepsilon,\delta'),~~N\geq N_{2}(\varepsilon,\delta'),~~M\geq M_{2}(\varepsilon,\delta'),$
	Then, for $K\geq\log\left(4/\left(\delta\,\mu^{*}\left(\delta;K_{2}^{*}\right)\right)\right)$,
	we have 
	\[
	\|{v}_{K}-v^{*}\|_{2,\,\rho}\leq2\tilde{\gamma}^{1/2}C_{\rho,\,\mu}^{1/2}\left(d_{2,\,\mu}\left(T\,\mathcal{F}\left(\Theta\right),\,\mathcal{F}\left(\Theta\right)\right)+2\,\varepsilon\right)
	\]
	with probability at least $1-\delta$. 
\end{theorem} 
\textit{Remarks.}
1. That is, if we choose enough samples $N_2$ of the states, enough samples
$M_2$ of the next state, and enough random samples $J_2$ of the parameter
$\theta$, and then for large enough number of iterations $K_2$, the
$\mathcal{L}_{2}$ error in the value function is determined by the
inherent Bellman error of the function class $\mathcal{F}\left(\Theta\right)$.
2. For the function families $\mathcal{F}(\Theta)$ discussed earlier (RPBF),
the inherent Bellman error, $d_{2,\,\mu}\left(T\,\mathcal{F}\left(\Theta\right),\,\mathcal{F}\left(\Theta\right)\right)=0$ indeed, and so the value function
will have small $\mathcal{L}_{2}$ error with high probability. {3. Note that the sample complexity bounds are independent of the state space dimension though the computational complexity of sampling from the state space does indeed depend on that dimension.}

\noindent Next we give a similar guarantee for $\mathcal{L}_{1}$
error for Algorithm \ref{algo:EVL+RPBF} by considering approximation
in $\mathcal{L}_{1,\,\mu}\left(\mathbb{S}\right).$ 

Denote 
\begin{eqnarray*}
	N_{1}(\varepsilon,\delta') & = & 2^{7}5^{2}\bar{v}_{\max}^{2}\log\left[\frac{40\,e\left(J_1+1\right)}{\delta}\left(10\,e\,\bar{v}_{\max}\right)^{J}\right],~~\\
	M_{1}(\varepsilon,\delta') & = & \left(\frac{\bar{v}_{\max}^{2}}{2}\right)\log\left[\frac{10\,N_1\,|\mathbb{A}|}{\delta'}\right],\\
	J_{1}(\varepsilon,\delta') & = & \left(\cfrac{5C}{\varepsilon}\left(1+\sqrt{2\log\cfrac{5}{\delta'}}\right)\right)^{2},\\
	K_{1}^{*} & = & \left\lceil \frac{\ln\left(C_{\rho,\,\mu}\varepsilon\right)-\ln\left(2\,v_{\max}\right)}{\ln\,\gamma}\right\rceil ,~~\text{and}\\
	\mu^{*}\left(p;\,K^{*}\right) & = & \left(1-p\right)p^{\left(K^{*}-1\right)},
\end{eqnarray*}
where $C$ is the same constant that appears in the definition of
$\mathcal{F}\left(\Theta\right)$ (see \cite{rahimi2009weighted})
and $\bar{v}_{max}=v_{\max}/\varepsilon$. Set $\delta':=1-\left(1-\delta/2\right)^{1/\left(K_{1}^{*}-1\right)}$.

\begin{theorem} \label{thm:RPBF_L1} Given an $\varepsilon>0$, and
	a $\delta\in\left(0,\,1\right)$, choose $J\geq J_{1}(\varepsilon,\delta'),~~N\geq N_{1}(\varepsilon,\delta'),~~M\geq M_{1}(\varepsilon,\delta'),$
	Then, for $K\geq\log\left(4/\left(\delta\,\mu^{*}\left(\delta;K_{1}^{*}\right)\right)\right)$,
	we have 
	\[
	\|{v}_{K}-v^{*}\|_{1,\,\rho}\leq2\,C_{\rho,\,\mu}\left(d_{1,\,\mu}\left(T\,\mathcal{F}\left(\Theta\right),\,\mathcal{F}\left(\Theta\right)\right)+2\,\varepsilon\right)
	\]
	with probability at least $1-\delta$. 
\end{theorem} 

\textit{Remarks.} 1. Again, note that the above result implies that the RBPF function family $\mathcal{F}(\Theta)$ has inherent
Bellman error $d_{1,\,\mu}\left(T\,\mathcal{F}\left(\Theta\right),\,\mathcal{F}\left(\Theta\right)\right)=0$, so that for 
enough samples $N_1$ of the states, enough samples
$M_1$ of the next state, and enough random samples $J_1$ of the parameter
$\theta$, and then for large enough number of iterations $K_1$, the
value function will have small $\mathcal{L}_{1}$ error with high
probability.  {2. As above, note that there is no dependence on state space dimension in the sample complexity bounds though computational complexity of sampling states from the state space indeed depends on it.}

\subsection{Non-parametric Function Approximation in RKHS}
\label{sec:rkhs}

We now consider non-parametric function approximation combined with EVL. We employ a Reproducing Kernel Hilbert Space (RKHS) for function
approximation since for suitably chosen kernels, it is dense in the space of continuous functions and hence has a `universal' function approximation property. In the RKHS setting, we can obtain guarantees directly with respect to the supremum norm.

We will consider a regularized RKHS setting with a continuous, symmetric
and positive semidefinite kernel $K:\mathbb{S}\times\mathbb{S}\rightarrow\mathbb{R}$
and a regularization constant $\lambda>0$. The RKHS space, $\mathcal{H}_{K}$
is defined to be the closure of the linear span of $\{K(s,\cdot)\}_{s\in\mathbb{S}}$
endowed with an inner product $\langle\cdot,\cdot\rangle_{\mathcal{H}_{K}}$.
The inner product $\langle\cdot,\,\cdot\rangle_{\mathcal{H}_{K}}$
for $\mathcal{H}_{K}$ is defined such that $\langle K\left(x,\,\cdot\right),\,K\left(y,\,\cdot\right)\rangle_{\mathcal{H}_{K}}=K\left(x,\,y\right)$
for all $x,\,y\in\mathbb{S}$, i.e., $\langle\sum_{i}\alpha_{i}K\left(x_{i},\,\cdot\right),\,\sum_{j}\beta_{j}K\left(y_{j},\,\cdot\right)\rangle_{\mathcal{H}_{K}}=\sum_{i,\,j}\alpha_{i}\beta_{j}K\left(x_{i},\,y_{j}\right)$.
Subsequently, the inner product satisfies the reproducing property:
$\langle K\left(s,\,\cdot\right),\,f\rangle_{\mathcal{H}_{K}}=f\left(s\right)$
for all $s\in\mathbb{S}$ and $f\in\mathcal{H}_{K}$. The corresponding
RKHS norm is defined in terms of the inner product $\|f\|_{\mathcal{H}_{K}}:=\sqrt{\langle f,\,f\rangle_{\mathcal{H}_{K}}}$.
We assume that our kernel $K$ is bounded so that $\kappa:=\sup_{s\in\mathbb{S}}\sqrt{K\left(s,\,s\right)}<\infty$.

To find the best fit $f\in\mathcal{H}_{K}$ to a function with data
$\left\{ \left(s_{n},\,\tilde{v}\left(s_{n}\right)\right)\right\} _{n=1}^{N}$,
we solve the regularized least squares problem: 
\begin{equation}
\min_{f\in\mathcal{H}_{K}}\left\{ \frac{1}{N}\sum_{n=1}^{N}\left(f\left(s_{n}\right)-\tilde{v}\left(s_{n}\right)\right)^{2}+\lambda\,\|f\|_{\mathcal{H}_{K}}^{2}\right\} .\label{opt:RKHS}
\end{equation}
This is a convex optimization problem (the norm squared is convex),
and has a closed form solution by the Representer Theorem. In particular,
the optimal solution is of the form $\hat{f}\left(s\right)=\sum_{n=1}^{N}\alpha_{n}K\left(s_{n},\,s\right)$
where the weights $\alpha^{1:N}=\left(\alpha_{1},\ldots,\,\alpha_{N}\right)$
are the solution to the linear system 
\begin{equation}
\left(\left[K\left(s_{i},\,s_{j}\right)\right]_{i,\,j=1}^{N}+\lambda\,N\,I\right)\left(\alpha_{n}\right)_{n=1}^{N}=\left(\tilde{v}\left(s_{n}\right)\right)_{n=1}^{N}.\label{eq:RKHS_linear}
\end{equation}
This yields EVL algorithm with randomized function fitting in a regularized RKHS (EVL+RKHS) displayed as Algorithm \ref{algo:EVL+RKHS}.

Note that the optimization problem in Step 3 in Algorithm \ref{algo:EVL+RKHS}
is analogous to the optimization problem in Step 3 of Algorithm \ref{algo:EVL+RPBF}
which finds an approximate best fit within the finite-dimensional
space $\widehat{\mathcal{F}}\left(\theta^{1:J}\right)$, rather than
the entire space $\mathcal{F}\left(\Theta\right)$, while Problem
(\ref{opt:RKHS}) in Algorithm \ref{algo:EVL+RKHS} optimizes over
the entire space $\mathcal{H}_{K}$. This difference can be reconciled
by the Representer Theorem, since it states that optimization over
$\mathcal{H}_{K}$ in Problem (\ref{opt:RKHS}) is equivalent to optimization
over the finite-dimensional space spanned by $\left\{ K\left(s_{n},\,\cdot\right)\text{ : }n=1,\ldots,\,N\right\} $.
Note that the regularization $\lambda\,\|f\|_{\mathcal{H}_{K}}^{2}$
is a requirement of the Representer Theorem.

\begin{algorithm}[H]
\caption{\label{algo:EVL+RKHS} EVL with regularized RKHS (EVL+RKHS)}
Input: probability distribution $\mu$ on $\mathbb{S}$;\\
 sample sizes $N\geq1,$ $M\geq1$; penalty $\lambda$;\\
 initial seed $v_{0}$; counter $k=0$.

For $k=1,\ldots,K$
\begin{enumerate}
\item Sample $\left\{ s_{n}\right\} _{n=1}^{N}\sim\mu$. 
\item Compute 
\[
\tilde{v}_{k}\left(s_{n}\right)=\min_{a\in\mathbb{A}}\left\{ c\left(s_{n},\,a\right)+\frac{\gamma}{M}\sum_{m=1}^{M}{v}_{k}\left(X_{m}^{s_{n},\,a}\right)\right\} ,
\]
where $\left\{ X_{m}^{s_{n},\,a}\right\} _{m=1}^{M}\sim Q\left(\cdot\,\vert\,s_{n},\,a\right)$
are i.i.d. 
\item ${v}_{k+1}(\cdot)$ is given by 
\[
\arg\min\limits _{f\in\mathcal{H}_{K}}\left\{ \cfrac{1}{N}\sum_{n=1}^{N}(f(s_{n})-\tilde{v}(s_{n}))^{2}+\lambda\lvert\lvert f\rvert\vert_{\mathcal{H}_{K}}\right\} .
\]
\item Increment $k\leftarrow k+1$ and return to Step 1. 
\end{enumerate}
\end{algorithm}

We define the \textit{regression function} $f_{M}\text{ : }\mathbb{S}\rightarrow\mathbb{R}$
via 
\[
f_{M}\left(s\right)\triangleq\mathbb{E}\left[\min_{a\in\mathbb{A}}\left\{ c\left(s,\,a\right)+\frac{\gamma}{M}\sum_{m=1}^{M}v\left(X_{m}^{s,\,a}\right)\right\} \right],\,\forall s\in\mathbb{S},
\]
it is the expected value of our empirical estimator of $T\,v$. As
expected, $f_{M}\rightarrow T\,v$ as $M\rightarrow\infty$. We note
that $f_{M}$ is not necessarily equal to $T\,v$ by Jensen's inequality.
We require the following assumption on $f_{M}$ to continue.

\begin{assumption} \label{assu:RKHS} For every $M\geq1$, $f_{M}\left(s\right)=\int_{\mathbb{S}}K\left(s,\,y\right)\alpha\left(y\right)\mu\left(dy\right)$
for some $\alpha\in\mathcal{L}_{2,\,\mu}\left(\mathbb{S}\right)$.
\end{assumption}

Regression functions play a key role in statistical learning theory,
Assumption \ref{assu:RKHS} states that the regression function lies
in the span of the kernel $K$. {It is satisfied whenever $K$ is a universal kernel. Some
examples of universal kernels follow.} Additionally, when $\mathcal{H}_{K}$
is dense in the space of Lipschitz functions, then the inherent Bellman error
is zero. For example, $K\left(s,\,s'\right)=\exp\left(-\gamma\,\|s-s'\|_{2}\right)$,
$K\left(s,\,s'\right)=1-\frac{1}{a}\|s-s'\|_{1}$, and $K\left(s,\,s'\right)=\exp\left(\gamma\,\|s-s'\|_{1}\right)$
are all universal kernels.

Denote 
\begin{eqnarray*}
N_{\infty}(\varepsilon,\delta') & = & \left(\frac{4\,C_{K}\kappa}{\varepsilon\left(1-\gamma\right)}\right)^{6}\log\left(\frac{4}{\delta'}\right)^{2}\\
M_{\infty}(\varepsilon) & = & \frac{160\,v_{\max}^{2}}{\left(\varepsilon\left(1-\gamma\right)\right)^{2}}\log\left(\frac{2\,|\mathbb{A}|\,\gamma\left(8\,v_{\max}-\varepsilon\left(1-\gamma\right)\right)}{\varepsilon\left(1-\gamma\right)\left(2-\gamma\right)}\right)\\
K_{\infty}^{*} & = & \left\lceil \frac{\ln\left(\varepsilon\right)-\ln\left(4\,v_{\max}\right)}{\ln\,\gamma}\right\rceil 
\end{eqnarray*}
where $C_{K}$ is a constant independent of the dimension of $\mathbb{S}$
(see \cite{smale2005shannon} for the details on how $C_{K}$ depends
on the kernel $K$) and set $\delta'=1-\left(1-\delta/2\right)^{1/\left(K_{\infty}^{*}-1\right)}$.
\begin{theorem} \label{thm:RKHS_Linfty} Suppose Assumption \ref{assu:RKHS}
holds. Given any $\varepsilon>0$ and $\delta\in\left(0,\,1\right)$,
choose an $N\geq N_{\infty}(\varepsilon,\delta')$ and an $M\geq M_{\infty}(\varepsilon)$.
Then, for any $K\geq\log\left(4/\left(\delta\,\mu^{*}\left(\delta;K_{\infty}^{*}\right)\right)\right)$,
\[
\|{v}_{K}-v^{*}\|_{\infty}\leq\varepsilon
\]
with probability at least $1-\delta$. \end{theorem}

Note that we provide guarantees on ${\mathcal L}_1$ and ${\mathcal L}_2$ error (can be generalized to ${\mathcal L}_p$) with the RPBF method and for ${\mathcal L}_{\infty}$ error with the RKHS-based randomized function fitting method. Getting guarantees for the ${\mathcal L}_{p}$ error with the RKHS method has proved quite difficult, as has bounds on the ${\mathcal L}_{\infty}$ error with the RBPF method.


\section{Analysis in a Random operator framework}

\label{sec:analysis}

We will analyze Algorithms \ref{algo:EVL+RPBF} and \ref{algo:EVL+RKHS}
in terms of random operators since this framework is general enough
to encompass many such algorithms. The reader can see that Step 2
of both algorithms involves iteration of the empirical Bellman operator
while Step 3 involves a randomized function fitting step which is
done differently and in different spaces in both algorithms. We use
random operator notation to write these algorithms in a compact way,
and then derive a clean and to a large-extent unified convergence
analysis. The key idea is to use the notion of stochastic dominance
to bound the error process with an easy to analyze ``dominating''
Markov chain. Then, we can infer the solution quality of our algorithms
via the probability distribution of the dominating Markov chain. This analysis idea refines (and in fact, simplifies) 
the idea we introduced in  \cite{haskell2016empirical}
for MDPs with finite state and action spaces (where there is no function
fitting) in the supremum norm. In this paper, we develop the technique
further, give a stronger convergence rate, account for randomized
function approximation, and also generalize the technique to $\mathcal{L}_{p}$
norms.

We introduce a probability space $\left(\Omega,\mathcal{B}\left(\Omega\right),P\right)$
on which to define random operators, where $\Omega$ is a sample space
with elements denoted $\omega\in\Omega$, $\mathcal{B}\left(\Omega\right)$
is the Borel $\sigma-$algebra on $\Omega$, and $P$ is a probability
distribution on $\left(\Omega,\,\mathcal{B}\left(\Omega\right)\right)$.
A random operator is an operator-valued random variable on $\left(\Omega,\,\mathcal{B}\left(\Omega\right),\,P\right)$.
We define the first random operator on $\mathcal{F}\left(\mathbb{S}\right)$
as $\widehat{T}(v)=\left(s_{n},\,\tilde{v}\left(s_{n}\right)\right)_{n=1}^{N}$
where $(s_{n})_{n=1}^{N}$ is chosen from $\mathbb{S}$ according
to a distribution $\mu\in\mathcal{M}\left(\mathbb{S}\right)$ and
\[
\tilde{v}\left(s_{n}\right)=\min_{a\in\mathbb{A}}\left\{ c\left(s_{n},\,a\right)+\frac{\gamma}{M}\sum_{m=1}^{M}v\left(X_{m}^{s_{n},\,a}\right)\right\} ,
\]
$n=1,\ldots,N$ is an approximation of $\left[T\,v\right]\left(s_{n}\right)$
for all $n=1,\ldots,N$. In other words, $\widehat{T}$ maps from
$v\in\mathcal{F}\left(\mathbb{S};\,v_{\max}\right)$ to a randomly
generated sample of $N$ input-output pairs $\left(s_{n},\,\tilde{v}\left(s_{n}\right)\right)_{n=1}^{N}$
of the function $T\,v$. Note that $\widehat{T}$ depends on sample
sizes $N$ and $M$. Next, we have the function reconstruction operator
$\widehat{\Pi}_{\mathcal{F}}$ which maps the data $\left(s_{n},\,\tilde{v}\left(s_{n}\right)\right)_{n=1}^{N}$
to an element in $\mathcal{F}$. Note that $\widehat{\Pi}_{\mathcal{F}}$
is not necessarily deterministic since Algorithms \ref{algo:EVL+RPBF}
and \ref{algo:EVL+RKHS} use randomized function fitting. We can now
write both algorithms succinctly as 
\begin{equation}
v_{k+1}=\widehat{G}\,v_{k}:=\widehat{\Pi}_{\mathcal{F}}\widehat{T}\,v_{k},\label{eq:EVL}
\end{equation}
which can be further written in terms of residual error $\varepsilon_{k}=\widehat{G}\,v_{k}-T\,v_{k}$
as 
\begin{equation}
v_{k+1}=\widehat{G}\,v_{k}=T\,v_{k}+\varepsilon_{k}.\label{eq:EVL_error}
\end{equation}
Iteration of these operators corresponds to repeated samples from
$\left(\Omega,\mathcal{B}(\Omega),P\right)$, so we define the space
of sequences $\left(\Omega^{\infty},\mathcal{B}(\Omega^{\infty}),\mathcal{P}\right)$
where $\Omega^{\infty}=\times_{k=0}^{\infty}\Omega$ with elements
denoted $\boldsymbol{\omega}=\left(\omega_{k}\right)_{k\geq0}$, $\mathcal{B}\left(\Omega^{\infty}\right)=\times_{k=0}^{\infty}\mathcal{B}\left(\Omega\right)$,
and $\mathcal{P}$ is the probability measure on $\left(\Omega^{\infty},\,\mathcal{B}\left(\Omega^{\infty}\right)\right)$
guaranteed by the Kolmogorov extension theorem applied to $\mathcal{P}$.

The random sequences $\left(v_{k}\right)_{k\geq0}$ in Algorithms
\ref{algo:EVL+RPBF} and \ref{algo:EVL+RKHS} given by 
\begin{eqnarray*}
v_{k+1} & = & \widehat{\Pi}_{\mathcal{F}}\widehat{T}\left(\omega_{k}\right)v_{k}\\
 & = & \widehat{\Pi}_{\mathcal{F}}\widehat{T}\left(\omega_{k}\right)\widehat{\Pi}_{\mathcal{F}}\widehat{T}\left(\omega_{k-1}\right)\cdots\widehat{\Pi}_{\mathcal{F}}\widehat{T}\left(\omega_{0}\right)v_{0},
\end{eqnarray*}
for all $k\geq0$ is a stochastic process defined on $\left(\Omega^{\infty},\mathcal{B}(\Omega^{\infty}),\mathcal{P}\right)$.
We now analyze error propagation over the iterations.

{Let us now bound how the Bellman residual at each iteration
of EVL is changing.} There have already been some results which address the error propagation both in $\mathcal{L}_{\infty}$
and $\mathcal{L}_{p}$ $(p\geq1)$ norms \cite{munos2007performance}.
After adapting \cite[Lemma 3]{munos2008finite}, we obtain the
following {$p$-norm error bounds} on $v_{K}-v^{*}$ in terms of the
errors $\left\{ \varepsilon_{k}\right\} _{k\geq0}$. 
\begin{lemma}
\label{lem:error_prop} For any $K\geq1$, and $\varepsilon>0$, suppose
$\|\varepsilon_{k}\|_{p,\,\mu}\leq\varepsilon$ for all $k=0,\,1,\ldots,\,K-1$,
then 
\begin{equation}
\|v_{K}-v^{*}\|_{p,\,\rho}\leq2\left(\frac{1-\gamma^{K+1}}{1-\gamma}\right)^{\frac{p-1}{p}}\left[C_{\rho,\,\mu}^{1/p}\varepsilon+\gamma^{K/p}\left(2\,v_{\max}\right)\right].\label{eq:Error_Lp}
\end{equation}
\end{lemma} 
where $C_{\rho,\,\mu}$ is as defined in Assumption \ref{assu:Random}.
Note that Lemma \ref{lem:error_prop} assumes that {$\|\varepsilon_{k}\|_{p,\,\mu}\leq\varepsilon$}
which we will show subsequently that it is true with high probability.

The second inequality is for the supremum norm. 
\begin{lemma} \label{lem:Error_supremum}
For any $K\geq1$ and $\varepsilon>0$, suppose $\|\varepsilon_{k}\|_{\infty}\leq\varepsilon$
for all $k=0,\,1,\ldots,\,K-1$, then 
\begin{equation}
\|v_{K}-v^{*}\|_{\infty}\leq\varepsilon/\left(1-\gamma\right)+\gamma^{K}\left(2\,v_{\max}\right).\label{eq:Error_supremum}
\end{equation}
\end{lemma} Inequalities (\ref{eq:Error_Lp}) and (\ref{eq:Error_supremum})
are the key to analyzing iteration of Equation (\ref{eq:EVL_error}).

\subsection{Convergence analysis using stochastic dominance}

We now provide a (unified) convergence analysis for iteration of a sequence of random operators given by (\ref{eq:EVL}) and (\ref{eq:EVL_error}). Later, we will show how it can be applied to Algorithms \ref{algo:EVL+RPBF} and \ref{algo:EVL+RKHS}. 
We will use $\|\cdot\|$ to denote a general norm in the following
discussion, since our idea applies to all instances of $p\in[1,\,\infty)$
and $p=\infty$ simultaneously. The magnitude of the error in iteration
$k\geq0$ is then $\|\varepsilon_{k}\|$. We make the following key
assumption for a general EVL algorithm. \begin{assumption} \label{assu:Random}
For $\varepsilon>0$, there is a $q\in\left(0,\,1\right)$ such that
$\text{Pr}\left\{ \|\varepsilon_{k}\|\leq\varepsilon\right\} \geq q$
for all $k\geq0$. \end{assumption} Assumption \ref{assu:Random}
states that we can find a lower bound on the probability of the event
$\left\{ \|\varepsilon_{k}\|\leq\varepsilon\right\} $ that is independent
of $k$ and $\left(v_{k}\right)_{k\geq0}$ (but does depend on $\varepsilon$).
Equivalently, we are giving a lower bound on the probability of the
event $\left\{ \|T\,v_{k}-\widehat{G}\,v_{k}\|\leq\varepsilon\right\} $.
This is possible for all of the algorithms that we proposed earlier.
In particular, we can control $q$ in Assumption \ref{assu:Random}
through the sample sizes in each iteration of EVL. Naturally, for
a given $\varepsilon$, $q$ increases as the number of samples grows.

We first choose $\varepsilon>0$ and the number of iterations $K^{*}$
for our EVL algorithms to reach a desired accuracy (this choice of
$K^{*}$ comes from the inequalities (\ref{eq:Error_Lp}) and (\ref{eq:Error_supremum})).
We call iteration $k$ ``good'' if the error $\|\varepsilon_{k}\|$
is within our desired tolerance $\varepsilon$ and ``bad'' when
the error is greater than our desired tolerance. We then construct
a stochastic process $(X_{k})_{k\geq0}$ on $\left(\Omega^{\infty},\,\mathcal{B}\left(\Omega^{\infty}\right),\,\mathcal{P}\right)$
with state space $\mathcal{K}:=\left\{ 1,\,\,2,\ldots,\,K^{*}\right\} $
such that 
\[
X_{k+1}=\begin{cases}
\max\left\{ X_{k}-1,\,1\right\} , & \text{if iteration \ensuremath{k} is "good"},\\
K^{*}, & \text{otherwise}.
\end{cases}
\]
The stochastic process $(X_{k})_{k\geq0}$ is easier to analyze than
$(v_{k})_{k\geq0}$ because it is defined on a finite state space,
however $(X_{k})_{k\geq0}$ is not necessarily a Markov chain.

We next construct a ``dominating'' Markov chain $(Y_{k})_{k\geq0}$
to help us analyze the behavior of $(X_{k})_{k\geq0}$. We construct
$(Y_{k})_{k\geq0}$ on $\left(\mathcal{K}^{\infty},\,\mathcal{B}\right)$,
the canonical measurable space of trajectories on $\mathcal{K}$,
so $Y_{k}\mbox{ : }\mathcal{K}^{\infty}\rightarrow\mathbb{R}$, and
we let $\mathcal{Q}$ denote the probability measure of $(Y_{k})_{k\geq0}$
on $\left(\mathbb{R}^{\infty},\,\mathcal{B}\right)$. Since $(Y_{k})_{k\geq0}$
will be a Markov chain by construction, the probability measure $\mathcal{Q}$
is completely determined by an initial distribution on $\mathbb{R}$
and a transition kernel for $(Y_{k})_{k\geq0}$. We always initialize
$Y_{0}=K^{*}$, and then construct the transition kernel as follows
\[
Y_{k+1}=\begin{cases}
\max\left\{ Y_{k}-1,\,1\right\} , & \mbox{w.p. }q,\\
K^{*}, & \mbox{w.p. }1-q,
\end{cases}
\]
where $q$ is the probability of a ``good'' iteration with respect
to the corresponding norm. Note that the $(Y_{k})_{k\geq0}$ we introduce
here is different and has much smaller state space than the one we
introduced in \cite{haskell2016empirical} leading to stronger convergence
guarantees.

We now describe a stochastic dominance relationship between the two
stochastic processes $(X_{k})_{k\geq0}$ and $(Y_{k})_{k\geq0}$.
We will establish that $(Y_{k})_{k\geq0}$ is ``larger'' than $(X_{k})_{k\geq0}$
in a stochastic sense. 
\begin{definition} Let $X$ and $Y$ be two
real-valued random variables, then $X$ is \textit{stochastically
dominated} by $Y$, written $X\leq_{st}Y$, when $\mathbb{E}\left[f\left(X\right)\right]\leq\mathbb{E}\left[f\left(Y\right)\right]$
for all increasing functions $f\mbox{ : }\mathbb{R}\rightarrow\mathbb{R}$.
Equivalently, $X\leq_{st}Y$ when $\mbox{Pr}\left\{ X\geq\theta\right\} \leq\mbox{Pr}\left\{ Y\geq\theta\right\} $
for all $\theta$ in the support of $Y$. 
\end{definition} 
Let $\left\{ \mathcal{F}_{k}\right\} _{k\geq0}$
be the filtration on $\left(\Omega^{\infty},\,\mathcal{B}\left(\Omega^{\infty}\right),\,\mathcal{P}\right)$
corresponding to the evolution of information about $(X_{k})_{k\geq0}$,
and let $\left[X_{k+1}\,\vert\,\mathcal{F}_{k}\right]$ denote the
conditional distribution of $X_{k+1}$ given the information $\mathcal{F}_{k}$.
We have the following initial results on the relationship between
$(X_{k})_{k\geq0}$ and $(Y_{k})_{k\geq0}$.

{The following theorem, our main result for our random operator analysis,
establishes the relationship between the stochastic process $\left\{ X_{k}\right\} _{k\geq0}$
and the Markov chain $\left\{ Y_{k}\right\} _{k\geq0}$. Under Assumption
3, this result allows us to bound the stochastic process $\left\{ X_{k}\right\} _{k\geq0}$
which keeps track of the error in EVL with the dominating Markov chain
$\left\{ Y_{k}\right\} _{k\geq0}$.}
\begin{theorem}
\label{thm:Random}
{Under Assumption \ref{assu:Random}:}\\
(i) $X_{k}\leq_{st}Y_{k}$ for all $k\geq0$.\\
(ii) $\mbox{Pr}\left\{ Y_{k}\leq\eta\right\} \leq\mbox{Pr}\left\{ X_{k}\leq\eta\right\} $
for any $\eta\in\mathbb{R}$ and all $k\geq0$. 
\end{theorem} 
{The proof is relegated to Appendix \ref{app:C}.} 
By Theorem \ref{thm:Random}, if $X_{K}\leq_{st}Y_{K}$ and we can
make $\mbox{Pr}\left\{ Y_{K}\leq\eta\right\} $ large, then we will
also obtain a meaningful bound on $\text{Pr}\left\{ X_{K}\leq\eta\right\} $.
Following this observation, the next two corollaries are the main
mechanisms for our general sample complexity results for EVL. 

{The following corollary follows from bounding the mixing time of the dominating Markov chain $\left\{ Y_{k}\right\} _{k\geq0}$
and employing our general $p-$norm error bound Lemma \ref{lem:error_prop}.}
\begin{corollary}
\label{cor:Random_Lp} For a given $p\in[1,\,\infty)$, and any $\varepsilon>0$,
and $\delta\in\left(0,\,1\right)$, suppose Assumption \ref{assu:Random}
holds for this $\varepsilon$, and choose any $K^{*}\geq1$. Then
for $q\geq\left(1/2+\delta/2\right)^{1/\left(K^{*}-1\right)}$ and
$K\geq\log\left(4/\left(\left(1/2-\delta/2\right)\left(1-q\right)q^{K^{*}-1}\right)\right),$
we have 
\[
\|v_{K}-v^{*}\|_{p,\,\rho}\leq2\left(\frac{1-\gamma^{K^{*}+1}}{1-\gamma}\right)^{\frac{p-1}{p}}\left[C_{\rho,\,\mu}^{1/p}\varepsilon+\gamma^{K^{*}/p}\left(2\,v_{\max}\right)\right]
\]
with probability at least $\delta$. 
\end{corollary}
{The proof is relegated to Appendix \ref{app:C}.}

{The next Corollary uses the same reasoning for the supremum norm case. It follows from bounding the mixing time of the dominating Markov chain $\left\{ Y_{k}\right\} _{k\geq0}$ and employing our general $\infty-$norm error bound Lemma \ref{lem:Error_supremum}.}
\begin{corollary}
\label{cor:Random_supremum} Given any $\varepsilon>0$ and $\delta\in\left(0,\,1\right)$,
suppose Assumption \ref{assu:Random} holds for this $\varepsilon$,
and choose any $K^{*}\geq1$. For $q\geq\left(1/2+\delta/2\right)^{1/\left(K^{*}-1\right)}$
and $K\geq\log\left(4/\left(\left(1/2-\delta/2\right)\left(1-q\right)q^{K^{*}-1}\right)\right),$
we have 
\[
\text{Pr}\left\{ \|v_{K}-v^{*}\|_{\infty}\leq\varepsilon/\left(1-\gamma\right)+\gamma^{K^{*}}\left(2\,v_{\max}\right)\right\} \geq\delta.
\]
\end{corollary} The sample complexity results for both EVL algorithms
from Section \ref{sec:algos} follow from Corollaries \ref{cor:Random_Lp}
and \ref{cor:Random_supremum}. This is shown next.


\subsection{Proofs of Theorems \ref{thm:RPBF_L2}, \ref{thm:RPBF_L1}, and \ref{thm:RKHS_Linfty}}

We now apply our random operator framework to both EVL algorithms.
We will see that it is easy to check the conditions of Corollaries
\ref{cor:Random_Lp} and \ref{cor:Random_supremum}, from which we
obtain specific sample complexity results. {We will use Theorems \ref{thm:Bellman_L2}, \ref{thm:Bellman_L1}, and \ref{thm:Supremum_Bellman} which are all ``one-step'' results which bound the error in a single step of Algorithm \ref{algo:EVL+RPBF} (in the 1- and 2-norm) and Algorithm \ref{algo:EVL+RKHS} (in the $\infty-$norm) compared to the true Bellman operator.}

We first give the proof of Theorem \ref{thm:RPBF_L2}. We let
$p\left(N,\,M,\,J,\,\varepsilon\right)$ denote the a lower bound
on the probability of the event $\left\{ \|\widehat{T}\,v-T\,v\|_{2,\,\mu}\leq\varepsilon\right\} $.
\begin{proof}{(}of Theorem \ref{thm:RPBF_L2}{)} Starting with
	inequality (\ref{eq:Error_Lp}) for $p=2$ and using the statement
	of Theorem \ref{thm:Bellman_L2} in Appendix \ref{app:C}, we have $\|v_{K}-v^{*}\|_{2,\,\rho}$
	\begin{eqnarray*}
		& \leq & 2\left(\frac{1}{1-\gamma}\right)^{1/2}C_{\rho,\,\mu}^{1/2}\left(d_{2,\,\mu}\left(T\,\mathcal{F}\left(\Theta\right),\,\mathcal{F}\left(\Theta\right)\right)+\varepsilon\right)\\
		&  & +4\left(\frac{1}{1-\gamma}\right)^{1/2}v_{\max}\gamma^{K/2},
	\end{eqnarray*}
	when $\|\varepsilon_{k}\|_{2,\,\mu}\leq d_{2,\,\mu}\left(T\,\mathcal{F}\left(\Theta\right),\,\mathcal{F}\left(\Theta\right)\right)+\varepsilon$
	for all $k=0,\,1,\ldots,\,K-1$. We choose $K^{*}\geq1$ to satisfy
	\[
	4\left(\frac{1}{1-\gamma}\right)^{1/2}v_{\max}\gamma^{K^{*}/2}\leq2\left(\frac{1}{1-\gamma}\right)^{1/2}C_{\rho,\,\mu}^{1/2}\varepsilon
	\]
	which implies $K^{*}=2\left\lceil \frac{\ln\left(C_{\rho,\,\mu}^{1/2}\varepsilon\right)-\ln\left(2\,v_{\max}\right)}{\ln\,\gamma}\right\rceil .$
	Based on Corollary \ref{cor:Random_Lp}, we just need to choose $N,\,M,\,J$
	such that $p\left(N,\,M,\,J,\,\varepsilon\right)\geq\left(1-\delta/2\right)^{1/\left(K^{*}-1\right)}$.
	We then apply the statement of Theorem \ref{thm:Bellman_L1} with
	$p=1-\left(1-\delta/2\right)^{1/\left(K^{*}-1\right)}$. \end{proof}

We  now give the proof of Theorem \ref{thm:RPBF_L1} along the same lines as for Theorem \ref{thm:RPBF_L1}. 
Let $p\left(N,\,M,\,J,\,\varepsilon\right)$ denote
the lower bound on the probability of the event $\left\{ \|\widehat{T}\,v-T\,v\|_{1,\,\mu}\leq\varepsilon\right\} $
for $\varepsilon>0$. We also note that $d_{1,\,\mu}\left(T\,v,\,\mathcal{F}\left(\Theta\right)\right)\leq d_{1,\,\mu}\left(T\,\mathcal{F}\left(\Theta\right),\,\mathcal{F}\left(\Theta\right)\right)$
for all $v\in\mathcal{F}\left(\Theta\right)$. 
\begin{proof}{(}of Theorem \ref{thm:RPBF_L1}{)} Starting with inequality (\ref{eq:Error_Lp})
for $p=1$ and using the statement of Theorem \ref{thm:Bellman_L1} in Appendix \ref{app:D},
we have $\|v_{K}-v^{*}\|_{1,\,\rho}$
\[
\leq2\,C_{\rho,\,\mu}\left(d_{1,\,\mu}\left(T\,\mathcal{F}\left(\Theta\right),\,\mathcal{F}\left(\Theta\right)\right)+\varepsilon\right)+4\,v_{\max}\gamma^{K}
\]
when $\|\varepsilon_{k}\|_{1,\,\mu}\leq d_{1,\,\mu}\left(T\,\mathcal{F}\left(\Theta\right),\,\mathcal{F}\left(\Theta\right)\right)+\varepsilon$
for all $k=0,\,1,\ldots,\,K-1$. Choose $K^{*}$ such that 
\[
4\,v_{\max}\gamma^{K}\leq2\,C_{\rho,\,\mu}\varepsilon\Rightarrow K^{*}=\left\lceil \frac{\ln\left(C_{\rho,\,\mu}\varepsilon\right)-\ln\left(2\,v_{\max}\right)}{\ln\,\gamma}\right\rceil .
\]
Based on Corollary \ref{cor:Random_Lp}, we just need to choose $N,\,M,\,J$
such that $p\left(N,\,M,\,J,\,\varepsilon\right)\geq\left(1-\delta/2\right)^{1/\left(K^{*}-1\right)}$.
We then apply the statement of Theorem \ref{thm:Bellman_L1} with
probability $1-\left(1-\delta/2\right)^{1/\left(K^{*}-1\right)}$.
\end{proof}

We now provide proof of $\mathcal{L}_{\infty}$ function fitting in
RKHS based on Theorem \ref{thm:Supremum_Bellman} in Appendix \ref{app:C}.
For this proof, we let $p\left(N,\,M,\,\varepsilon\right)$ denote
a lower bound on the probability of the event $\left\{ \|\widehat{T}\,v-T\,v\|_{\infty}\leq\varepsilon\right\} $.
\begin{proof}{(}of Theorem \ref{thm:RKHS_Linfty}{)} By inequality
(\ref{eq:Error_supremum}), we choose $\varepsilon$ and $K^{*}\geq1$
such that $\varepsilon/\left(1-\gamma\right)\leq\epsilon/2$ and $\gamma^{K^{*}}\left(2\,v_{\max}\right)\leq\epsilon/2$
by setting 
\[
K^{*}\geq\left\lceil \frac{\ln\left(\epsilon\right)-\ln\left(4\,v_{\max}\right)}{\ln\left(\gamma\right)}\right\rceil .
\]
Based on Corollary \ref{cor:Random_Lp}, we next choose $N$ and $M$
such that $p\left(N,\,M,\,\varepsilon\right)\geq\left(1-\delta/2\right)^{1/\left(K^{*}-1\right)}$.
We then apply the statement of Theorem \ref{thm:Supremum_Bellman}
with error $\epsilon\left(1-\gamma\right)/2$ and probability $1-\left(1-\delta/2\right)^{1/\left(K^{*}-1\right)}$.
\end{proof}

\section{Numerical Experiments}\label{sec:numerical}

We now present numerical performance of our algorithm by testing it on the benchmark optimal replacement problem \cite{rust1997using,munos2008finite}.  The setting is that a product (such as a car) becomes more costly to maintain with time/miles, and must be replaced it some point.  Here, the state $s_t \in \mathbb{R}_+$ represents the accumulated utilization of the product. Thus, $s_t=0$ denotes a brand new durable good. Here, $\mathbb{A}=\{0,1\}$, so at each time step, $t$, we can either replace the product $(a_t = 0)$ or keep it $(a_t = 1)$.  Replacement incurs a cost $C$ while keeping the product has a maintenance  cost, $c(s_t)$, associated with it. The transition probabilities are as follows:
\begin{equation*}
q(s_{t+1}|s_t,a_t) =
\begin{cases}
\lambda e^{-\lambda(s_{t+1}-s_t)}, & \text{if } s_{t+1} \geq s_t \text{ and } a_t=1,\\
\lambda e^{-\lambda s_{t+1}}, & \text{if } s_{t+1} \geq 0 \text{ and } a_t=0, \text{and} \\
0, & \text{otherwise}
\end{cases}
\end{equation*}
and the reward function is given by
\begin{equation*}
r(s_t,a_t) =
\begin{cases}
-c(s_t), & \text{if } a_t=1, \text{and}\\
-C-c(0), & \text{if } a_t=0.\\
\end{cases}
\end{equation*}

For our computation, we use $\gamma=0.6, \lambda=0.5, C=30$ and $c(s)=4s$. The optimal value function and the optimal policy can be computed analytically for this problem.
For EVL+RPBF, we use $J$ random parameterized Fourier functions $\{\phi(s, \theta_j) = \cos(\theta_j^Ts+b)\}_{j=1}^J$ with $\theta_j \sim \mathcal{N}(0,0.01)$ and $b \sim \text{Unif}[-\pi, \pi]$. We fix J=5. For EVL+RKHS, we use Gaussian kernel defined as $k(x,y) = \exp(\lvert \lvert x-y \rvert \rvert^2/ (2\sigma^2))$ with $1/\sigma^2 = 0.01 $ and $\mathcal{L}_2$ regularization. We fix the regularization coefficient to be $10^{-2}$. The underlying function space for FVI is polynomials of degree 4. The results are plotted after 20 iterations. 

The error in each iteration for different algorithms with $N=100$ states and $M=5$ is shown in Figure 1. On Y-axis, it shows the relative error computed as $\sup_{s \in \mathbb{S}}\vert v^*(s)-v^{\pi_k}(s)/v^*(s)\vert$ with iterations $k$ on the X-axis. It shows that EVL+RPBF has relative error below 10\% after 20 iterations. {FVI is close to it but EVL+RKHS has larger relative error though it may improve with a higher $M$ or by using other kernels.} This is also reflected in the actual runtime performance: EVL+RPBF takes 8,705s, FVI 8,654s and EVL+RKHS takes 42,173s to get within 0.1 relative error. The computational complexity of kernel methods increases quadratically with number of samples and needs a matrix inversion resulting in a slower perfomance. 
\begin{figure}[t]
\begin{center}	\includegraphics[width=0.45\textwidth]{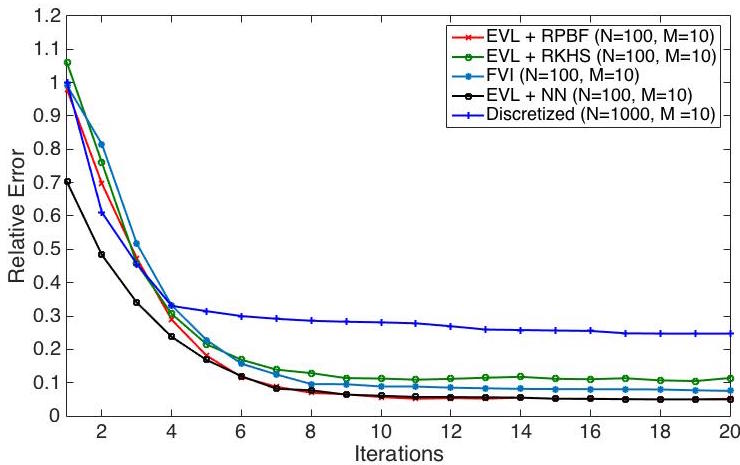}
	\caption{Relative Error with iterations for various algorithms. }
	\end{center}
		\label{fig:error}
\end{figure}

\begin{table}[b]
	\begin{center}
		\begin{tabular}[b]{cccc}
			\hline 
			Goal  & EVL+RPBF   & FVI  & EVL+RKHS \tabularnewline
			\hline 
			50  & 5.4m  & 4.8m & 8.7m \tabularnewline
			100  & 18.3m  & 23.7m & 32.1m \tabularnewline
			150  & 36.7m  & 41.5m & 54.3m \tabularnewline
			\hline 
		\end{tabular}
		\caption{Runtime performance of various algorithms on the cart-pole problem {(m=minutes)}}
		\label{tab:cartpole-runtime} %
	\end{center}
\end{table}

Note that performance of FVI depends on being able to choose suitable basis functions which for the optimal replacement problem is easy. For other problems, we may expect both EVL algorithms to perform better. So, we tested the algorithms on the cart-pole balancing problem, another benchmark problem but for which the optimal value function is unknown.  We formulate it as a continuous $4$-dimensional state space with $2$ action MDP. The state comprises of the position of the cart,$x$, velocity of the cart, $\dot{x}$, angle of the pole in radians, $\theta$  and the angular velocity of the pole, $\dot{\theta}$ . The actions are to add a force of $-10N$ or $+10N$ to the cart, pushing it left or right. We add $\pm50\%$ noise to these actions. 
{For system dynamics, let $m_c$ and $m_p$ be the mass of  cart and pole respectively. Let $l$ be the length of the pole. If $F_t$ is the force applied to the cart at time $t$, then acceleration of pole is 
$$\ddot{\theta_t} = \cfrac{g \sin \theta_t + \cos \theta_t \left(\cfrac{-F_t- m_pl \dot{\theta_t}^2\sin \theta_t}{m_c+m_p} \right)}{l\left(\cfrac{4}{3}-\cfrac{m_p cos^2 \theta_t}{m_c+m_p} \right)}$$  
and acceleration of cart is
$$\ddot{x_t} = \cfrac{F_t + m_p l \left(\dot{\theta_t}^2\sin \theta_t - \ddot{\theta_t} \cos \theta_t \right)}{m_c+m_p}. $$
Now let $\tau$ be the time step for Euler's method, we have the following state transition equations:
\begin{align*}
x_{t+1} &= x_t + \tau \dot{x_t}\\
\dot{x}_{t+1} &= \dot{x}_{t} + \tau \ddot{x_t}\\
\theta_{t+1} &= x_t + \tau \dot{\theta_t}\\
\dot{\theta}_{t+1} &= \dot{\theta}_{t} + \tau \ddot{\theta_t}
\end{align*}
}
Rewards are zero except for failure state (if the position of cart reaches beyond $\pm2.4$, or the pole exceeds an angle of $\pm12$ degrees), it is $-1$. For our experiments, we choose $N=100$ and $M=1$. In case of RPBF, we consider parameterized Fourier basis of the form $\text{cos}(\textbf{w}^{T}\textbf{s}+b)$ where $\textbf{w}=[w_{1},w_{2}]$, $w_{1},w_{2}\sim\mathcal{N}(0,1)$ and $b\sim\text{Unif}[-\pi,\pi]$. We fix $J=10$ for our EVL+RPBF. For RKHS, we consider Gaussian kernel, $K(s_{1},s_{2})=\exp\left(-\sigma{||s_{1}-s_{2}||^{2}}/2\right)$ with $\sigma=0.01$. We limit each episode to 1000 time steps. We compute the average length of the episode for which we are able to balance the pole without hitting the failure state. This is the goal in Table \ref{tab:cartpole-runtime}. The other columns show run-time needed for the algorithms to learn to achieve such a goal. 

From the table, we can see that EVL+RPBF outperforms FVI and EVL+RKHS. Note that guarantees for FVI are only available for ${\mathcal L}_2$-error and for EVL-RPBF for ${\mathcal L}_p$-error. EVL-RKHS is the only algorithm that can provide guarantees on the sup-norm error. Also note that when for problems for which the value functions are not so regular, and good basis functions difficult to guess, the EVL+RKHS method is likely to perform better but as of now we do not have a numerical example to demonstrate this.

{We also tested our algorithms on the Acrobot problem, a 2-link pendulum with only the second joint actuated. The objective is to swing the end-effector to a height which is at least the length of one link above the base starting with both links pointing downwards. The state here is six dimensional which are $\sin(\cdot)$ and $\cos(\cdot)$ of the two rotational joint angles and the joint angular velocities. There are three actions available: +1, 0 or -1, corresponding to the torque on the joint between the two pendulum links. We modify the environment available from \texttt{OpenAI} by injecting a uniform noise in the actions so that the transitions are not deterministic. The reward is 1 if the goal state is reached, else 0. We choose $N=2000, M=1, J=100$. Fig. \ref{fig:error_acro} represents the reward for both of the proposed algorithms. Not only does EVL+RPBF perform better, it is also faster than EVL+RKHS by an average of 3.67 minutes per iteration. The reason for this is that the EVL+RKHS algorithm is designed to provide guarantees on sup-error, a much more stringent requirement than the $L_p$-error that EVL+RPBF algorithm provides guarantees on.}

\begin{figure}[t]
\begin{center}	\includegraphics[width=0.5\textwidth]{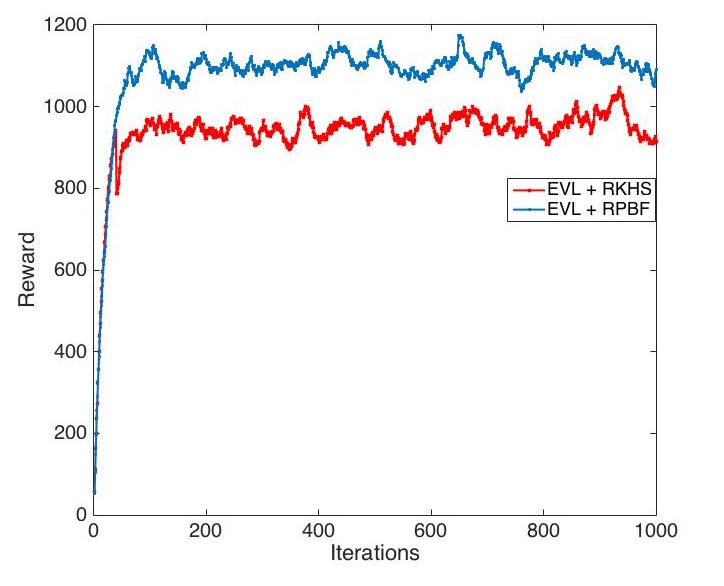}
	\caption{Performance on the Acrobot problem}
	\end{center}
		\label{fig:error_acro}
\end{figure}

\section{Conclusion}

\label{sec:conclusion}

In this paper, we have introduced universally applicable approximate
dynamic programming algorithms for continuous state space MDPs {with finite action spaces}. The
algorithms introduced are based on using randomization {to improve computational tractability and reduce}
 the `curse of dimensionality' via the synthesis of the `random function
approximation' and `empirical' approaches. Our first algorithm is
based on a random parametric function fitting by sampling parameters
in each iteration. The second is based on sampling states which then
yield a set of basis functions in an RKHS from the kernel. Both function
fitting steps involve convex optimization problems and can be implemented
with standard packages. Both algorithms can be viewed as iteration
of a type of random Bellman operator followed by a random projection
operator. 
{Iterated random operators in general are difficult to analyze.
Nevertheless, we can construct Markov chains that stochastically dominate the error sequences which
simplify the analysis \cite{haskell2016empirical}. In fact, the introduced method may be viewed as a `probabilistic contraction analysis' method in contrast to stochastic Lyapunov techniques and other methods for analyzing stochastic iterative algorithms.}
They yield convergence but also non-asymptotic sample complexity bounds. Numerical experiments on the cart-pole balancing and the Acrobat
problems suggests good performance in practice. More rigorous numerical
analysis will be conducted as part of future work.


\bibliographystyle{IEEEtran}
\bibliography{References-TAC17edp}

\begin{thebibliography}{10}
\providecommand{\url}[1]{#1}
\csname url@samestyle\endcsname
\providecommand{\newblock}{\relax}
\providecommand{\bibinfo}[2]{#2}
\providecommand{\BIBentrySTDinterwordspacing}{\spaceskip=0pt\relax}
\providecommand{\BIBentryALTinterwordstretchfactor}{4}
\providecommand{\BIBentryALTinterwordspacing}{\spaceskip=\fontdimen2\font plus
\BIBentryALTinterwordstretchfactor\fontdimen3\font minus
  \fontdimen4\font\relax}
\providecommand{\BIBforeignlanguage}[2]{{%
\expandafter\ifx\csname l@#1\endcsname\relax
\typeout{** WARNING: IEEEtran.bst: No hyphenation pattern has been}%
\typeout{** loaded for the language `#1'. Using the pattern for}%
\typeout{** the default language instead.}%
\else
\language=\csname l@#1\endcsname
\fi
#2}}
\providecommand{\BIBdecl}{\relax}
\BIBdecl

\bibitem{haskell2017randomized}
W.~B. Haskell, P.~Yu, H.~Sharma, and R.~Jain, ``Randomized function
  fitting-based empirical value iteration,'' in \emph{Decision and Control
  (CDC), 2017 IEEE 56th Annual Conference on}.\hskip 1em plus 0.5em minus
  0.4em\relax IEEE, 2017, pp. 2467--2472.

\bibitem{bertsekas2011dynamic}
D.~P. Bertsekas, ``Dynamic programming and optimal control 3rd edition, volume
  ii,'' \emph{Belmont, MA: Athena Scientific}, 2011.

\bibitem{powell2007approximate}
W.~B. Powell, \emph{Approximate Dynamic Programming: Solving the curses of
  dimensionality}.\hskip 1em plus 0.5em minus 0.4em\relax John Wiley \& Sons,
  2007, vol. 703.

\bibitem{sutton1998reinforcement}
R.~S. Sutton and A.~G. Barto, \emph{Reinforcement learning: An
  introduction}.\hskip 1em plus 0.5em minus 0.4em\relax Cambridge Univ Press,
  1998, vol.~1, no.~1.

\bibitem{devore1998nonlinear}
R.~A. DeVore, ``Nonlinear approximation,'' \emph{Acta numerica}, vol.~7, pp.
  51--150, 1998.

\bibitem{Bertsekas_ADP_2005}
\BIBentryALTinterwordspacing
D.~P. Bertsekas, ``Dynamic programming and suboptimal control: A survey from
  adp to mpc,'' 2005. [Online]. Available:
  \url{http://citeseerx.ist.psu.edu/viewdoc/summary?doi=10.1.1.68.4541}
\BIBentrySTDinterwordspacing

\bibitem{Bertsekas_ADP_2010}
------, \emph{Dynamic Programming and Optimal Control}, 2010.

\bibitem{Powell_ADP_2007}
W.~B. Powell, \emph{Approximate Dynamic Programming: Solving the Curses of
  Dimensionality (Wiley Series in Probability and Statistics)}.\hskip 1em plus
  0.5em minus 0.4em\relax Wiley-Interscience, 2007.

\bibitem{munos2008finite}
R.~Munos and C.~Szepesv{\'a}ri, ``Finite-time bounds for fitted value
  iteration,'' \emph{The Journal of Machine Learning Research}, vol.~9, pp.
  815--857, 2008.

\bibitem{haskell2016empirical}
W.~B. Haskell, R.~Jain, and D.~Kalathil, ``Empirical dynamic programming,''
  \emph{Mathematics of Operations Research}, vol.~41, no.~2, pp. 402--429,
  2016.

\bibitem{rust1997using}
J.~Rust, ``Using randomization to break the curse of dimensionality,''
  \emph{Econometrica: Journal of the Econometric Society}, pp. 487--516, 1997.

\bibitem{de2004constraint}
D.~P. De~Farias and B.~Van~Roy, ``On constraint sampling in the linear
  programming approach to approximate dynamic programming,'' \emph{Mathematics
  of operations research}, vol.~29, no.~3, pp. 462--478, 2004.

\bibitem{ormoneit2002kernel}
D.~Ormoneit and {\'S}.~Sen, ``Kernel-based reinforcement learning,''
  \emph{Machine learning}, vol.~49, no. 2-3, pp. 161--178, 2002.

\bibitem{grunewalder2012modelling}
S.~Grunewalder, G.~Lever, L.~Baldassarre, M.~Pontil, and A.~Gretton,
  ``Modelling transition dynamics in mdps with rkhs embeddings,'' \emph{arXiv
  preprint arXiv:1206.4655}, 2012.

\bibitem{szepesvari2001efficient}
C.~Szepesv{\'a}ri, ``Efficient approximate planning in continuous space
  markovian decision problems,'' \emph{AI Communications}, vol.~14, no.~3, pp.
  163--176, 2001.

\bibitem{munos2007performance}
R.~Munos, ``Performance bounds in l\_p-norm for approximate value iteration,''
  \emph{SIAM journal on control and optimization}, vol.~46, no.~2, pp.
  541--561, 2007.

\bibitem{de2003linear}
D.~P. De~Farias and B.~Van~Roy, ``The linear programming approach to
  approximate dynamic programming,'' \emph{Operations research}, vol.~51,
  no.~6, pp. 850--865, 2003.

\bibitem{bhat2012non}
N.~Bhat, V.~Farias, and C.~C. Moallemi, ``Non-parametric approximate dynamic
  programming via the kernel method,'' in \emph{Advances in Neural Information
  Processing Systems}, 2012, pp. 386--394.

\bibitem{munos2003error}
R.~Munos, ``Error bounds for approximate policy iteration,'' in \emph{ICML},
  vol.~3, 2003, pp. 560--567.

\bibitem{konda2000actor}
V.~R. Konda and J.~N. Tsitsiklis, ``Actor-critic algorithms,'' in
  \emph{Advances in neural information processing systems}, 2000, pp.
  1008--1014.

\bibitem{jain2010simulation}
R.~Jain and P.~Varaiya, ``Simulation-based optimization of markov decision
  processes: An empirical process theory approach,'' \emph{Automatica},
  vol.~46, no.~8, pp. 1297--1304, 2010.

\bibitem{sutton2000policy}
R.~S. Sutton, D.~A. McAllester, S.~P. Singh, and Y.~Mansour, ``Policy gradient
  methods for reinforcement learning with function approximation,'' in
  \emph{Advances in neural information processing systems}, 2000, pp.
  1057--1063.

\bibitem{peters2016policy}
J.~Peters and J.~A. Bagnell, ``Policy gradient methods,'' \emph{Encyclopedia of
  Machine Learning and Data Mining}, pp. 1--4, 2016.

\bibitem{mnih2015human}
V.~Mnih, K.~Kavukcuoglu, D.~Silver, A.~A. Rusu, J.~Veness, M.~G. Bellemare,
  A.~Graves, M.~Riedmiller, A.~K. Fidjeland, G.~Ostrovski \emph{et~al.},
  ``Human-level control through deep reinforcement learning,'' \emph{Nature},
  vol. 518, no. 7540, p. 529, 2015.

\bibitem{rahimi2008uniform}
A.~Rahimi and B.~Recht, ``Uniform approximation of functions with random
  bases,'' in \emph{Communication, Control, and Computing, 2008 46th Annual
  Allerton Conference on}.\hskip 1em plus 0.5em minus 0.4em\relax IEEE, 2008,
  pp. 555--561.

\bibitem{rahimi2009weighted}
------, ``Weighted sums of random kitchen sinks: Replacing minimization with
  randomization in learning,'' in \emph{Advances in neural information
  processing systems}, 2009, pp. 1313--1320.

\bibitem{puterman2014markov}
M.~L. Puterman, \emph{Markov decision processes: discrete stochastic dynamic
  programming}.\hskip 1em plus 0.5em minus 0.4em\relax John Wiley \& Sons,
  2014.

\bibitem{smale2005shannon}
S.~Smale and D.-X. Zhou, ``Shannon sampling ii: Connections to learning
  theory,'' \emph{Applied and Computational Harmonic Analysis}, vol.~19, no.~3,
  pp. 285--302, 2005.

\bibitem{Shaked2007}
M.~Shaked and J.~G. Shanthikumar, \emph{Stochastic Orders}.\hskip 1em plus
  0.5em minus 0.4em\relax Springer, 2007.

\bibitem{Levin_Mixing_2008}
D.~A. Levin, Y.~Peres, and E.~L. Wilmer, \emph{Markov Chains and Mixing
  Times}.\hskip 1em plus 0.5em minus 0.4em\relax American Mathematical Society,
  2008.

\bibitem{anthony2009neural}
M.~Anthony and P.~L. Bartlett, \emph{Neural network learning: Theoretical
  foundations}.\hskip 1em plus 0.5em minus 0.4em\relax cambridge university
  press, 2009.

\bibitem{haussler1995sphere}
D.~Haussler, ``Sphere packing numbers for subsets of the boolean n-cube with
  bounded vapnik-chervonenkis dimension,'' \emph{Journal of Combinatorial
  Theory, Series A}, vol.~69, no.~2, pp. 217--232, 1995.

\end{thebibliography}

\appendix

\subsection{Supplement for Section \ref{sec:algos}}\label{app:A}

The following computation shows that $T$ maps bounded functions to
Lipschitz continuous functions when $Q$ and $c$ are both Lipschitz
continuous in the sense of (\ref{eq:smooth}) and (\ref{eq:smooth-1}).
Suppose $\|v\|_{\infty}\leq v_{\max}$, then $T\,v$ is Lipschitz
continuous with constant $L_{c}+\gamma\,v_{\max}L_{Q}$. We have 
\begin{align*}
 & |\left[T\,v\right]\left(s\right)-\left[T\,v\right]\left(s'\right)|\\
 & \leq\,\max_{a\in\mathbb{A}}|c\left(s,\,a\right)-c\left(s',\,a\right)|\\
 & +\gamma\,\max_{a\in\mathbb{A}}|\int v\left(y\right)Q\left(dy\,\vert\,s,\,a\right)-\int v\left(y\right)Q\left(dy\,\vert\,s',\,a\right)|\\
\leq\, & L_{c}\|s-s'\|_{2}+\gamma\,v_{\max}\max_{a\in\mathbb{A}}\int|Q\left(dy\,\vert\,s,\,a\right)-Q\left(dy\,\vert\,s',\,a\right)|\\
\leq\, & \left(L_{c}+\gamma\,v_{\max}L_{Q}\right)\|s-s'\|_{2}.
\end{align*}

\subsection{Supplement for Section \ref{sec:analysis}}\label{app:B}

First, we need to adapt \cite[Lemma 3]{munos2008finite} to obtain
point-wise error bounds on $v_{K}-v^{*}$ in terms of the errors $\left\{ \varepsilon_{k}\right\} _{k\geq0}$.
These bounds are especially useful when analyzing the performance
of EVL with respect to other norms besides the supremum norm, since
$T$ does not have a contractive property with respect to any other
norm.

For any $\pi\in\Pi$, we define the operator $Q^{\pi}\text{ : }\mathcal{F}\left(\mathbb{S}\right)\rightarrow\mathcal{F}\left(\mathbb{S}\right)$
(which gives the transition mapping as a function of $\pi$) via 
\[
\left(Q^{\pi}v\right)\left(s\right)\triangleq\int_{\mathbb{S}}v\left(y\right)Q\left(dy\,\vert\,s,\,\pi\left(s\right)\right),\,\forall s\in\mathbb{S}.
\]
Then we define the operator $T^{\pi}\text{ : }\mathcal{F}\left(\mathbb{S}\right)\rightarrow\mathcal{F}\left(\mathbb{S}\right)$
via 
\[
\left[T^{\pi}v\right]\left(s\right)\triangleq c\left(s,\,\pi\left(s\right)\right)+\gamma\,\int_{\mathbb{S}}v\left(x\right)Q\left(dx\,\vert\,s,\,\pi\left(s\right)\right),\,\forall s\in\mathbb{S}.
\]
For later use, we let $\pi^{*}\in\Pi$ be an optimal policy satisfying
\[
\pi^{*}\left(s\right)\in\arg\min_{a\in\mathbb{A}}\left\{ c\left(s,\,a\right)+\gamma\,\int_{\mathbb{S}}v^{*}\left(x\right)Q\left(dx\,\vert\,s,\,a\right)\right\} ,\,
\]
$\forall s\in\mathbb{S},$, i.e., it is greedy with respect to $v^{*}$.
More generally, a policy $\pi\in\Pi$ is greedy with respect to $v\in\mathcal{F}\left(\mathbb{S}\right)$
if $T^{\pi}v=T\,v$.

For use throughout this section, we let $\pi_{k}$ be a greedy policy
with respect to $v_{k}$ so that $T^{\pi_{k}}v_{k}=T\,v_{k}$ for
all $k\geq0$. Then, for fixed $K\geq1$ we define the operators 
\begin{align*}
A_{K}\triangleq\, & \frac{1}{2}\left[\left(Q^{\pi^{*}}\right)^{K}+Q^{\pi_{K-1}}Q^{\pi_{K-2}}\cdots Q^{\pi_{0}}\right],\\
A_{k}\triangleq\, & \frac{1}{2}\left[\left(Q^{\pi^{*}}\right)^{K-k-1}+Q^{\pi_{K-1}}Q^{\pi_{K-2}}\cdots Q^{\pi_{k+1}}\right],
\end{align*}
for $k=0,\ldots,\,K-1$, formed by composition of transition kernels.
We let $\vec{1}$ be the constant function equal to one on $\mathbb{S}$,
and we define the constant $\tilde{\gamma}=\frac{2\left(1-\gamma^{K+1}\right)}{1-\gamma}$
for use shortly. We note that $\left\{ A_{k}\right\} _{k=0}^{K}$
are all linear operators and $A_{k}\vec{1}=\vec{1}$ for all $k=0,\ldots,\,K$.
\begin{lemma} \label{lem:point-wise} For any $K\geq1$,

(i) $v_{K}-v^{*}\leq\sum_{k=0}^{K-1}\gamma^{K-k-1}\left(Q^{\pi^{*}}\right)^{K-k-1}\varepsilon_{k}+\gamma^{K}\left(Q^{\pi^{*}}\right)^{K}\left(v_{0}-v^{*}\right)$.

(ii) $v_{K}-v^{*}\geq\sum_{k=0}^{K-1}\gamma^{K-k-1}\left(Q^{\pi_{K-1}}Q^{\pi_{K-2}}\cdots Q^{\pi_{k+1}}\right)\varepsilon_{k}+\gamma^{K}\left(Q^{\pi_{K-1}}Q^{\pi_{K-2}}\cdots Q^{\pi_{0}}\right)\left(v_{0}-v^{*}\right).$

(iii) $|v_{K}-v^{*}|\leq2\left[\sum_{k=0}^{K-1}\gamma^{K-k-1}A_{k}\,|\varepsilon_{k}|+\gamma^{K}A_{K}\left(2\,v_{\max}\right)\right]$.
\end{lemma} \begin{proof} (i) For any $k\geq1$, we have $T\,v_{k}\leq T^{\pi^{*}}v_{k}$
and $T^{\pi^{*}}v_{k}-T^{\pi^{*}}v^{*}=\gamma\,Q^{\pi^{*}}\left(v_{k}-v^{*}\right)$,
so $v_{k+1}-v^{*}=$ 
\[
T\,v_{k}+\varepsilon_{k}-T^{\pi^{*}}v_{k}+T^{\pi^{*}}v_{k}-T^{\pi^{*}}v^{*}\leq\gamma\,Q^{\pi^{*}}\left(v_{k}-v^{*}\right)+\varepsilon_{k}.
\]
The result then follows by induction.

(ii) Similarly, for any $k\geq1$, we have $T\,v^{*}\leq T^{\pi_{k}}v^{*}$
and $T\,v_{k}-T^{\pi_{k}}v^{*}=T^{\pi_{k}}v_{k}-T^{\pi_{k}}v^{*}=\gamma\,Q^{\pi_{k}}\left(v_{k}-v^{*}\right)$,
so $v_{k+1}-v^{*}=$ 
\[
T\,v_{k}+\varepsilon_{k}-T^{\pi_{k}}v^{*}+T^{\pi_{k}}v^{*}-T\,v^{*}\geq\gamma\,Q^{\pi_{k}}\left(v_{k}-v^{*}\right)+\varepsilon_{k}.
\]
Again, the result follows by induction.

(iii) If $f\leq g\leq h$ in $\mathcal{F}\left(\mathbb{S}\right)$,
then $|g|\leq|f|+|h|$, so combining parts (i) and (ii) gives

\[
|v_{K}-v^{*}|\leq2\sum_{k=0}^{K-1}\gamma^{K-k-1}A_{k}|\varepsilon_{k}|+2\,\gamma^{K}A_{K}|v_{0}-v^{*}|.
\]
Then we note that $|v_{0}-v^{*}|\leq2\,v_{\max}.$ \end{proof} Now
we use Lemma \ref{lem:point-wise} to derive $p-$norm bounds. 
\begin{proof}{(}of Lemma \ref{lem:error_prop}{)} Using $\sum_{k=0}^{K}\gamma^{k}=\left(1-\gamma^{K+1}\right)/\left(1-\gamma\right)$,
we define the constants 
\begin{align*}
\alpha_{k}=\, & \frac{\left(1-\gamma\right)\gamma^{K-k-1}}{1-\gamma^{K+1}},\,\forall k=0,\ldots,\,K-1,\\
\alpha_{K}=\, & \frac{\left(1-\gamma\right)\gamma^{K}}{1-\gamma^{K+1}},
\end{align*}
and we note that $\sum_{k=0}^{K}\alpha_{k}=1$. Then, we obtain $|v_{K}-v^{*}|$
\[
\leq\tilde{\gamma}\left[\sum_{k=0}^{K-1}\alpha_{k}A_{k}\,|\varepsilon_{k}|+\alpha_{K}A_{K}\left(2\,v_{\max}\right)\right]
\]
from Lemma \ref{lem:point-wise}(iii). Next, we compute $\|v_{K}-v^{*}\|_{p,\,\rho}^{p}=\,$
\begin{align*}
 & \int_{\mathbb{S}}|v_{K}\left(s\right)-v^{*}\left(s\right)|^{p}\rho\left(ds\right)\\
\leq\, & \tilde{\gamma}^{p}\int_{\mathbb{S}}\left[\sum_{k=0}^{K-1}\alpha_{k}A_{k}\,|\varepsilon_{k}|+\alpha_{K}A_{K}\left(2\,v_{\max}\right)\vec{1}\right]^{p}\left(s\right)\rho\left(ds\right)\\
\leq\, & \tilde{\gamma}^{p}\int_{\mathbb{S}}\left[\sum_{k=0}^{K-1}\alpha_{k}A_{k}\,|\varepsilon_{k}|^{p}+\alpha_{K}A_{K}\left(2\,v_{\max}\right)^{p}\vec{1}\right]\left(s\right)\rho\left(ds\right),
\end{align*}
using Jensen's inequality and convexity of $x\rightarrow|x|^{p}$.
Now, we have $\rho\,A_{k}\leq c_{\rho,\,\mu}\left(K-k-1\right)\mu$
for $k=0,\ldots,\,K-1$ by Assumption \ref{assu:Kernel}(ii) and so
for all $k=0,\ldots,\,K-1$, 
\[
\int_{\mathbb{S}}\left[A_{k}\,|\varepsilon_{k}|^{p}\right]\left(s\right)\rho\left(ds\right)\leq c_{\rho,\,\mu}\left(K-k-1\right)\|\varepsilon_{k}\|_{p,\,\mu}^{p}.
\]
We arrive at $\|v_{K}-v^{*}\|_{p,\,\rho}^{p}$ 
\begin{align*}
 & \leq\,\tilde{\gamma}^{p}\left[\sum_{k=0}^{K-1}\alpha_{k}c_{\rho,\,\mu}\left(K-k-1\right)\|\varepsilon_{k}\|_{p,\,\mu}^{p}+\alpha_{K}\left(2\,v_{\max}\right)^{p}\right]\\
=\, & 2^{p}\tilde{\gamma}^{p-1}\bigg[\sum_{k=0}^{K-1}\gamma^{K-k-1}c_{\rho,\,\mu}\left(K-k-1\right)\|\varepsilon_{k}\|_{p,\,\mu}^{p}\\
&~~~~~~~~~~~ +\gamma^{K}\left(2\,v_{\max}\right)^{p}\bigg],
\end{align*}
where we use $|v_{0}-v^{*}|^{p}\leq\left(2\,v_{\max}\right)^{p}$.
Now, by subadditivity of $x\rightarrow|x|^{t}$ for $t=1/p\in(0,\,1]$
with $p\in[1,\,\infty)$, assumption that $\|\varepsilon_{k}\|_{p,\,\mu}\leq\varepsilon$
for all $k=0,\,1,\ldots,\,K-1$, and since $\sum_{k=0}^{K-1}\gamma^{K-k-1}c_{\rho,\,\mu}\left(K-k-1\right)\leq C_{\rho,\,\mu}$
by Assumption \ref{assu:Kernel}(ii), we see 
\[
\|v_{K}-v^{*}\|_{p,\,\rho}\leq2\left(\frac{1-\gamma^{K+1}}{1-\gamma}\right)^{\frac{p-1}{p}}\left[C_{\rho,\,\mu}^{1/p}\varepsilon+\gamma^{K/p}\left(2\,v_{\max}\right)\right],
\]
which gives the desired result. \end{proof} Supremum norm error bounds
follow more easily from Lemma \ref{lem:point-wise}. \begin{proof}{(}Proof
of Lemma \ref{lem:Error_supremum}{)} We have 
\[
\|v_{K}-v^{*}\|_{\infty}\leq\,\max\{\|\sum_{k=0}^{K-1}\gamma^{K-k-1}\left(Q^{\pi^{*}}\right)^{K-k-1}\varepsilon_{k}
\]
\[
+\gamma^{K}\left(Q^{\pi^{*}}\right)^{K}\left(v_{0}-v^{*}\right)\|_{\infty},
\]
\[
\left.\|\sum_{k=0}^{K-1}\gamma^{K-k-1}\left(Q^{\pi_{K-1}}Q^{\pi_{K-2}}\cdots Q^{\pi_{K+1}}\right)\varepsilon_{k}\right.
\]
\[
+\gamma^{K}\left(Q^{\pi_{K-1}}Q^{\pi_{K-2}}\cdots Q^{\pi_{0}}\right)\left(v_{0}-v^{*}\right)\|_{\infty}\},
\]
by Lemma \ref{lem:point-wise}. Now, 
\[
\|\sum_{k=0}^{K-1}\gamma^{K-k-1}\left(Q^{\pi^{*}}\right)^{K-k-1}\varepsilon_{k}
\]
\[
+\gamma^{K}\left(Q^{\pi^{*}}\right)^{K}\left(v_{0}-v^{*}\right)\|_{\infty}
\]
\[
\leq\sum_{k=0}^{K-1}\gamma^{K-k-1}\|\varepsilon_{k}\|_{\infty}+\gamma^{K}\|v_{0}-v^{*}\|_{\infty}
\]
and 
\[
\|\sum_{k=0}^{K-1}\gamma^{K-k-1}\left(Q^{\pi_{K-1}}Q^{\pi_{K-2}}\cdots Q^{\pi_{K+1}}\right)\varepsilon_{k}
\]
\[
+\gamma^{K}\left(Q^{\pi_{K-1}}Q^{\pi_{K-2}}\cdots Q^{\pi_{0}}\right)\left(v_{0}-v^{*}\right)\|_{\infty}
\]
\[
\leq\,\sum_{k=0}^{K-1}\gamma^{K-k-1}\|\varepsilon_{k}\|_{\infty}+\gamma^{K}\|v_{0}-v^{*}\|_{\infty},
\]
where we use the triangle inequality, the fact that $|\left(Q\,f\right)\left(s\right)|\leq\int_{\mathbb{S}}|f\left(y\right)|Q\left(dy\,\vert\,s\right)\leq\|f\|_{\infty}$
for any transition kernel $Q$ on $\mathbb{S}$ and $f\in\mathcal{F}\left(\mathbb{S}\right)$,
and $|v_{0}-v^{*}|\leq2\,v_{\max}$. For any $K\geq1$, 
\begin{equation}
\|v_{K}-v^{*}\|_{\infty}\leq\sum_{k=0}^{K-1}\gamma^{K-k-1}\|\varepsilon_{k}\|_{\infty}+\gamma^{K}\left(2\,v_{\max}\right).\label{eq:Basic_supremum}
\end{equation}
Follows immediately since $\sum_{k=0}^{K-1}\gamma^{K-k-1}\varepsilon\leq\varepsilon/\left(1-\gamma\right)$
for all $K\geq1$. \end{proof} We emphasize that Lemma \ref{lem:Error_supremum}
does not require any assumptions on the transition probabilities,
in contrast to Lemma \ref{lem:error_prop} which requires Assumption
\ref{assu:Kernel}(ii).

\subsection{Supplement for Section \ref{sec:analysis}: Function approximation}\label{app:C}

We record several pertinent results here on the type of function reconstruction
used in our EVL algorithms. The first lemma is illustrative of approximation
results in Hilbert spaces, it gives an $O\left(1/\sqrt{J}\right)$
convergence rate on the error from using $\widehat{\mathcal{F}}\left(\theta^{1:J}\right)$
compared to $\mathcal{F}\left(\Theta\right)$ in $\mathcal{L}_{2,\,\mu}\left(\mathbb{S}\right)$
in probability. \begin{lemma} \cite[Lemma 1]{rahimi2009weighted}
Fix $f^{*}\in\mathcal{F}\left(\Theta\right)$, for any $\delta\in\left(0,\,1\right)$
there exists a function $\hat{f}\in\widehat{\mathcal{F}}\left(\theta^{1:J}\right)$
such that
\[
\|f^{*}-\hat{f}\|_{2,\,\mu}\leq\frac{C}{\sqrt{J}}\left(1+\sqrt{2\,\log\frac{1}{\delta}}\right)
\]
with probability at least $1-\delta$. \end{lemma}

The next result is an easy consequence of \cite[Lemma 1]{rahimi2009weighted}
and bounds the error from using $\widehat{\mathcal{F}}\left(\theta^{1:J}\right)$
compared to $\mathcal{F}\left(\Theta\right)$ in $\mathcal{L}_{1,\,\mu}\left(\mathbb{S}\right)$.
\begin{lemma} \label{lem:Approximation_L1} Fix $f^{*}\in\mathcal{F}\left(\Theta\right)$,
for any $\delta\in\left(0,\,1\right)$ there exists a function $\hat{f}\in\widehat{\mathcal{F}}\left(\theta^{1:J}\right)$
such that
\[
\|\hat{f}-f^{*}\|_{1,\,\mu}\leq\frac{C}{\sqrt{J}}\left(1+\sqrt{2\,\log\frac{1}{\delta}}\right)
\]
with probability at least $1-\delta$. \end{lemma} \begin{proof}
Choose $f,\,g\in\mathcal{F}\left(\mathbb{S}\right)$, then by Jensen's
inequality we have $\|f-g\|_{1,\,\mu}=\mathbb{E}_{\mu}\left[|f\left(S\right)-g\left(S\right)|\right]$
\[
=\mathbb{E}_{\mu}\left[\left(\left(f\left(S\right)-g\left(S\right)\right)^{2}\right)^{1/2}\right]\leq\sqrt{\mathbb{E}_{\mu}\left[\left(f\left(S\right)-g\left(S\right)\right)^{2}\right]}.
\]
The desired result then follows by \cite[Lemma 1]{rahimi2009weighted}.
\end{proof}

Now we consider function approximation in the supremum norm. Recall
the definition of the regression function 
\[
f_{M}\left(s\right)=\mathbb{E}\left[\min_{a\in\mathbb{A}}\left\{ c\left(s,\,a\right)+\frac{\gamma}{M}\sum_{m=1}^{M}v\left(X_{m}^{s,\,a}\right)\right\} \right],
\]
$\forall s\in\mathbb{S}$. Then we have the following approximation
result, for which we recall the constant $\kappa:=\sup_{s\in\mathbb{S}}\sqrt{K\left(s,\,s\right)}$.
\begin{corollary} \cite[Corollary 5]{smale2005shannon} For any $\delta\in\left(0,\,1\right)$,
\[
\|f_{z,\,\lambda}-f_{M}\|_{\mathcal{H}_{K}}
\]
\[
\leq\tilde{C}\,\kappa\left(\frac{\log\left(4/\delta\right)^{2}}{N}\right)^{1/6}\text{ for }\lambda=\left(\frac{\log\left(4/\delta\right)^{2}}{N}\right)^{1/3},
\]
with probability at least $1-\delta$. \end{corollary} \begin{proof}
Uses the fact that for any $f\in\mathcal{H}_{K}$, $\|f\|_{\infty}\leq\kappa\,\|f\|_{\mathcal{H}_{K}}$.
For any $s\in\mathbb{S}$, we have $|f\left(s\right)|=|\langle K\left(s,\,\cdot\right),\,f\left(\cdot\right)\rangle_{\mathcal{H}_{K}}|$
and subsequently 
\begin{align*}
|\langle K\left(s,\,\cdot\right),\,f\left(\cdot\right)\rangle_{\mathcal{H}_{K}}|\leq\, & \|K\left(s,\,\cdot\right)\|_{\mathcal{H}_{K}}\|f\|_{\mathcal{H}_{K}}\\
=\, & \sqrt{\langle K\left(s,\,\cdot\right),\,K\left(s,\,\cdot\right)\rangle{}_{\mathcal{H}_{K}}}\|f\|_{\mathcal{H}_{K}}\\
=\, & \sqrt{K\left(s,\,s\right)}\|f\|_{\mathcal{H}_{K}}\\
\leq\, & \sup_{s\in\mathbb{S}}\sqrt{K\left(s,\,s\right)}\|f\|_{\mathcal{H}_{K}},
\end{align*}
where the first inequality is by Cauchy-Schwartz and the second is
by assumption that $K$ is a bounded kernel. \end{proof} The preceding
result is about the error when approximating the regression function
$f_{M}$, but $f_{M}$ generally is not equal to $T\,v$. We bound
the error between $f_{M}$ and $T\,v$ as well in the next subsection.

\begin{proof} (Theorem \ref{thm:Random})
	First we note that, by \cite[Lemma A.1]{haskell2016empirical}, $\left[Y_{k+1}\,\vert\,Y_{k}=\eta\right]$
	is stochastically increasing in $\eta$ for all $k\geq0$, i.e. $\left[Y_{k+1}\,\vert\,Y_{k}=\eta\right]\leq_{st}\left[Y_{k+1}\,\vert\,Y_{k}=\eta'\right]$
	for all $\eta\leq\eta'$. Then, by \cite[Lemma A.2]{haskell2016empirical},
	$\left[X_{k+1}\,\vert\,X_{k}=\eta,\,\mathcal{F}_{k}\right]\leq_{st}\left[Y_{k+1}\,\vert\,Y_{k}=\eta\right]$
	for all $\eta\in f$ and $\mathcal{F}_{k}$ for all $k\geq0$.
	
	(i) Trivially, $X_{0}\leq_{st}Y_{0}$ since $X_{0}\leq_{as}Y_{0}$.
	Next, we see that $X_{1}\leq_{st}Y_{1}$ by \cite[Lemma A.1]{haskell2016empirical}.
	We prove the general case by induction. Suppose $X_{k}\leq_{st}Y_{k}$
	for $k\geq1$, and for this proof define the random variable 
	\[
	\mathcal{Y}\left(\theta\right)=\begin{cases}
	\max\left\{ \theta-1,\,1\right\} , & \mbox{w.p. }q,\\
	K^{*}, & \mbox{w.p. }1-q.
	\end{cases}
	\]
	to be the conditional distribution of $Y_{k}$ conditional on $\theta$,
	as a function of $\theta$. We see that $Y_{k+1}$ has the same distribution
	as $\left[\mathcal{Y}\left(\theta\right)\,\vert\,\theta=Y_{k}\right]$
	by definition. Since $\mathcal{Y}\left(\theta\right)$ are stochastically
	increasing by Lemma \cite[Lemma A.1]{haskell2016empirical}, we see that
	$\left[\mathcal{Y}\left(\theta\right)\,\vert\,\theta=Y_{k}\right]\geq_{st}\,\left[\mathcal{Y}\left(\theta\right)\,\vert\,\theta=X_{k}\right]$
	by \cite[Theorem 1.A.6]{Shaked2007} and our induction hypothesis.
	Now, $\left[\mathcal{Y}\left(\theta\right)\,\vert\,\theta=X_{k}\right]\geq_{st}\left[X_{k+1}\,\vert\,X_{k},\,\mathcal{F}_{k}\right]$
	by \cite[Theorem 1.A.3(d)]{Shaked2007} and Lemma \cite[Lemma A.2]{haskell2016empirical}
	for all histories $\mathcal{F}_{k}$. It follows that $Y_{k+1}\geq_{st}X_{k+1}$
	by transitivity of $\geq_{st}$.
	
	(ii) Follows from part (i) by the definition of $\leq_{st}$. 
\end{proof}

\begin{proof} (Corollary \ref{cor:Random_Lp}) Since $(Y_{k})_{k\geq0}$ is an irreducible Markov
	chain on a finite state space, its steady state distribution $\mu=(\mu\left(i\right))_{i=1}^{K^{*}}$
	on $\mathcal{K}$ exists. By \cite[Lemma 4.3]{haskell2016empirical},
	the steady state distribution of $(Y_{k})_{k\geq0}$ is $\mu=(\mu\left(i\right))_{i=1}^{K^{*}}$
	given by: 
	\begin{align*}
	\mu\left(1\right)=\, & q^{K^{*}-1}\\
	\mu\left(i\right)=\, & \left(1-q\right)q^{K^{*}-i}, & \forall i=2,\ldots,K^{*}-1,\\
	\mu\left(K^{*}\right)=\, & 1-q.
	\end{align*}
	The constant 
	\[
	\mu_{\min}\left(q;\,K^{*}\right):=\min\left\{ q^{K^{*}-1},\,\left(1-q\right)q^{\left(K^{*}-2\right)},\,\left(1-q\right)\right\} ,
	\]
	for all $q\in\left(0,\,1\right)$ and $K^{*}\geq1$, which is the
	minimum of the steady state probabilities appears shortly in the Markov
	chain mixing time bound for $(Y_{k})_{k\geq0}$. We note that $\mu^{*}\left(q;\,K^{*}\right)=\left(1-q\right)q^{K^{*}-1}\leq\mu_{\min}\left(q;\,K^{*}\right)$
	is a simple lower bound for $\mu_{\min}\left(q;\,K^{*}\right)$ (we
	defined $\mu^{*}\left(q;\,K^{*}\right)=\left(1-q\right)q^{K^{*}-1}$
	earlier).
	
	Now, recall that $\|\mu-\nu\|_{TV}=\frac{1}{2}\sum_{\eta=1}^{K^{*}}|\mu\left(\eta\right)-\nu\left(\eta\right)|$
	is the total variation distance for probability distributions on $\mathcal{K}$.
	Let $Q^{k}$ be the marginal distribution of $Y_{k}$ for $k\geq0$.
	By a Markov chain mixing time argument, e.g., \cite[Theorem 12.3]{Levin_Mixing_2008},
	we have that 
	\begin{eqnarray*}
		t_{\text{mix}}\left(\delta'\right) & := & \min\left\{ k\geq0\text{ : }\|Q^{k}-\mu\|_{TV}\leq\delta'\right\} \\
		& \leq & \log\left(\frac{1}{\delta'\mu_{\min}\left(q;\,K^{*}\right)}\right)\\
		& \leq & \log\left(\frac{1}{\delta'\left(1-q\right)q^{K^{*}-1}}\right)
	\end{eqnarray*}
	for any $\delta'\in\left(0,\,1\right)$. So, for $K\geq\log\left(1/\left(\delta'\left(1-q\right)q^{K^{*}-1}\right)\right)$
	we have $|\text{Pr}\left\{ Y_{K}=1\right\} -\mu\left(1\right)|=$
	\[
	|\text{Pr}\left\{ Y_{K}=1\right\} -q^{K^{*}-1}|\leq2\,\|Q^{K}-\mu\|_{TV}\leq2\,\delta',
	\]
	where we use $\mu\left(1\right)=q^{K^{*}-1}$. By Theorem \ref{thm:Random},
	$\text{Pr}\left\{ X_{K}=1\right\} \geq\text{Pr}\left\{ Y_{K}=1\right\} $
	and so 
	\[
	\text{Pr}\left\{ X_{K}=1\right\} \geq q^{K^{*}-1}-2\,\delta'.
	\]
	Choose $q$ and $\delta'$ to satisfy $q^{K^{*}-1}=1/2+\delta/2$
	and $2\,\delta'=q^{K^{*}-1}-\delta=1/2-\delta/2$ to get $q^{K^{*}-1}-2\,\delta'\geq\delta$,
	and the desired result follows. 
\end{proof}

\subsection{Bellman error}\label{app:D}

The layout of this subsection is modeled after the arguments in \cite{munos2008finite},
but with the added consideration of randomized function fitting. We use the following easy-to-establish fact. 
\begin{fact}
\label{fact:Basic} Let $X$ be a given set, and $f_{1}\mbox{ : }X\rightarrow\mathbb{R}$
and $f_{2}\mbox{ : }X\rightarrow\mathbb{R}$ be two real-valued functions
on $X$. Then,

(i) $|\inf_{x\in X}f_{1}\left(x\right)-\inf_{x\in X}f_{2}\left(x\right)|\leq\sup_{x\in X}|f_{1}\left(x\right)-f_{2}\left(x\right)|$,
and

(ii) $|\sup_{x\in X}f_{1}\left(x\right)-\sup_{x\in X}f_{2}\left(x\right)|\leq\sup_{x\in X}|f_{1}\left(x\right)-f_{2}\left(x\right)|$. 
\end{fact}
For example, Fact \ref{fact:Basic} can be used to show that $T$
is contractive in the supremum norm.

The next result is about $\widehat{T}$, it uses Hoeffding's inequality
to bound the estimation error between $\left\{ \tilde{v}\left(s_{n}\right)\right\} _{n=1}^{N}$
and $\left\{ \left[T\,v\right]\left(s_{n}\right)\right\} _{n=1}^{N}$
in probability. \begin{lemma} \label{lem:Bellman_Hoeffding} For
any $p\in\left[1,\infty\right]$, $f,\,v\in\mathcal{F}\left(\mathbb{S};\,v_{\max}\right)$,
and $\varepsilon>0$, 
\[
\text{Pr}\left\{ \left|\|f-T\,v\|_{p,\,\hat{\mu}}-\|f-\tilde{v}\|_{p,\,\hat{\mu}}\right|>\varepsilon\right\} 
\]
\[
\leq2\,N\,|\mathbb{A}|\,\exp\left(\frac{-2\,M\,\varepsilon^{2}}{v_{\max}^{2}}\right).
\]
\end{lemma} \begin{proof} First we have $\left|\|f-T\,v\|_{p,\,\hat{\mu}}-\|f-\tilde{v}\|_{p,\,\hat{\mu}}\right|\leq\|T\,v-\tilde{v}\|_{p,\,\hat{\mu}}$
by the reverse triangle inequality. Then, for any $s\in\mathbb{S}$
we have $\left|\left[T\,v\right]\left(s\right)-\tilde{v}\left(s\right)\right|=\,$
\[
\max_{a\in\mathbb{A}}|\left\{ c\left(s,\,a\right)+\gamma\,\int_{\mathbb{S}}v\left(x\right)Q\left(dx\,\vert\,s,\,a\right)\right\} 
\]
\[
-\left\{ c\left(s,\,a\right)+\frac{\gamma}{M}\sum_{m=1}^{M}v\left(X_{m}^{s,\,a}\right)\right\} |
\]
\[
\leq\,\gamma\,\max_{a\in\mathbb{A}}\left|\int_{\mathbb{S}}v\left(x\right)Q\left(dx\,\vert\,s,\,a\right)-\frac{1}{M}\sum_{m=1}^{M}v\left(X_{m}^{s,\,a}\right)\right|
\]
by Fact \ref{fact:Basic}. We may also take $v\left(s\right)\in\left[0,\,v_{\max}\right]$
for all $s\in\mathbb{S}$ by assumption on the cost function, so by
the Hoeffding inequality and the union bound we obtain 
\[
\text{Pr}\left\{ \max_{n=1,\ldots,\,N}\left|\left[T\,v\right]\left(s_{n}\right)-\tilde{v}\left(s_{n}\right)\right|\geq\varepsilon\right\} 
\]
\[
\leq2\,N\,|\mathbb{A}|\,\exp\left(\frac{-2\,M\,\varepsilon^{2}}{v_{\max^{2}}}\right)
\]
and thus 
\begin{eqnarray*}
 & \text{Pr}\left\{ \|T\,v-\tilde{v}\|_{p,\,\hat{\mu}}\geq\varepsilon\right\} \\
= & \text{Pr}\left\{ \left(\frac{1}{N}\sum_{n=1}^{N}\left|\left[T\,v\right]\left(s_{n}\right)-\tilde{v}\left(s_{n}\right)\right|^{p}\right)^{1/p}\geq\varepsilon\right\} \\
\leq & \text{Pr}\left\{ \max_{n=1,\ldots,\,N}\left|\left[T\,v\right]\left(s_{n}\right)-\tilde{v}\left(s_{n}\right)\right|\geq\varepsilon\right\} ,
\end{eqnarray*}
which gives the desired result. \end{proof} To continue, we introduce
the following additional notation corresponding to a set of functions
$\mathcal{F}\subset\mathcal{F}\left(\mathbb{S}\right)$: 
\begin{itemize}
\item $\mathcal{F}\left(s^{1:N}\right)\triangleq\left\{ \left(f\left(s_{1}\right),\ldots,\,f\left(s_{N}\right)\right)\text{ : }f\in\mathcal{F}\right\} $; 
\item $\mathcal{N}\left(\varepsilon,\,\mathcal{F}\left(s^{1:N}\right)\right)$
is the $\varepsilon-$covering number of $\mathcal{F}\left(s^{1:N}\right)$
with respect to the $1-$norm on $\mathbb{R}^{N}$. 
\end{itemize}
The next lemma uniformly bounds the estimation error between the true
expectation and the empirical expectation over the set $\widehat{\mathcal{F}}\left(\theta^{1:J}\right)$
(in the following statement, $e$ is Euler's number). \begin{lemma}
\label{lem:Bellman_empirical} For any $\varepsilon>0$ and $N\geq1$,
\[
\text{Pr}\left\{ \sup_{f\in\widehat{\mathcal{F}}\left(\theta^{1:J}\right)}\left|\frac{1}{N}\sum_{n=1}^{N}f\left(S_{n}\right)-\mathbb{E}_{\mu}\left[f\left(S\right)\right]\right|>\varepsilon\right\} 
\]
\[
\leq8\,e\left(J+1\right)\left(\frac{2\,e\,v_{\max}}{\varepsilon}\right)^{J}\exp\left(\frac{-N\,\varepsilon^{2}}{128\,v_{\max}^{2}}\right).
\]
\end{lemma} \begin{proof} For any $\mathcal{F}\subset\mathcal{F}\left(\mathbb{S};\,v_{\max}\right)$,
$\varepsilon>0$, and $N\geq1$, we have

\[
\text{Pr}\left\{ \sup_{f\in\widehat{\mathcal{F}}\left(\theta^{1:J}\right)}\left|\frac{1}{N}\sum_{n=1}^{N}f\left(S_{n}\right)-\mathbb{E}_{\mu}\left[f\left(S\right)\right]\right|>\varepsilon\right\} 
\]
\[
\leq8\,\mathbb{E}\left[\mathcal{N}\left(\varepsilon/8,\,\widehat{\mathcal{F}}\left(\theta^{1:J}\right)\left(s^{1:N}\right)\right)\right]\,\exp\left(\frac{-N\,\varepsilon^{2}}{128\,v_{\max}^{2}}\right).
\]
It remains to bound $\mathbb{E}\left[\mathcal{N}\left(\varepsilon/8,\,\widehat{\mathcal{F}}\left(\theta^{1:J}\right)\left(s^{1:N}\right)\right)\right]$.
We note that $\widehat{\mathcal{F}}\left(\theta^{1:J}\right)$ is
a subset of 
\[
\left\{ f\left(\cdot\right)=\sum_{j=1}^{J}\alpha_{j}\phi\left(\cdot;\,\theta_{j}\right)\text{ : }\left(\alpha_{1},\ldots,\,\alpha_{J}\right)\in\mathbb{R}^{J}\right\} ,
\]
which is a vector space with dimension $J$. By \cite[Corollary 11.5]{anthony2009neural},
the pseudo-dimension of $\widehat{\mathcal{F}}\left(\theta^{1:J}\right)$
is bounded above by $J$. Furthermore, 
\[
\mathcal{N}\left(\varepsilon,\,\widehat{\mathcal{F}}\left(\theta^{1:J}\right)\left(s^{1:N}\right)\right)\leq e\left(J+1\right)\left(\frac{2\,e\,v_{\max}}{\varepsilon}\right)^{J}
\]
by \cite[Corollary 3]{haussler1995sphere} which gives the desired
result. \end{proof}

To continue, we let $v'=v'\left(v,\,N,\,M,\,J,\,\mu,\,\nu\right)$
denote the (random) output of one iteration of EVL applied to $v\in\mathcal{F}\left(\mathbb{S}\right)$
as a function of the parameters $N,\,M,\,J\geq1$ and the probability
distributions $\mu$ and $\nu$. The next theorem bounds the error
between $T\,v$ and $v'$ in one iteration of EVL with respect to
$\mathcal{L}_{1,\,\mu}\left(\mathbb{S}\right)$, it is a direct adaptation
of \cite[Lemma 1]{munos2008finite} modified to account for the randomized
function fitting and the effective function space being $\mathcal{F}\left(\theta^{1:J}\right)$.
\begin{theorem} \label{thm:Bellman_L1} Choose $v\in\mathcal{F}\left(\mathbb{S};\,v_{\max}\right)$,
$\varepsilon>0$, and $\delta\in\left(0,\,1\right)$. Also choose
$J\geq\left[\frac{5\,C}{\varepsilon}\left(1+\sqrt{2\,\log\frac{5}{\delta}}\right)\right]^{2}$,
$N\geq2^{7}5^{2}\bar{v}_{\max}^{2}\log\left[\frac{40\,e\left(J+1\right)}{\delta}\left(10\,e\,\bar{v}_{\max}\right)^{J}\right]$,
and $M\geq\left(\frac{v_{\max}^{2}}{2\,\varepsilon^{2}}\right)\log\left[\frac{10\,N\,|\mathbb{A}|}{\delta}\right]$.
Then, for\\
 $v'=v'\left(v,\,N,\,M,\,J,\,\mu,\,\nu\right)$ we have $\|v'-T\,v\|_{1,\,\mu}\leq d_{1,\,\mu}\left(T\,v,\,\mathcal{F}\left(\Theta\right)\right)+\varepsilon$
with probability at least $1-\delta$. \end{theorem} \begin{proof}
Let $\varepsilon'>0$ be arbitrary and choose $f^{*}\in\mathcal{F}\left(\Theta\right)$
such that $\|f^{*}-T\,v\|_{1,\,\mu}\leq\inf_{f\in\mathcal{F}\left(\Theta\right)}\|f-T\,v\|_{1,\,\mu}+\varepsilon'$.
Then, choose $\hat{f}\in\widehat{\mathcal{F}}\left(\theta^{1:J}\right)$
such that $\|\hat{f}-T\,v\|_{1,\,\mu}\leq\|f^{*}-T\,v\|_{1,\,\mu}+\varepsilon/5$
with probability at least $1-\delta/5$ by Lemma \ref{lem:Approximation_L1}
by choosing $J\geq1$ to satisfy 
\[
\frac{C}{\sqrt{J}}\left(1+\sqrt{2\,\log\frac{1}{\left(\delta/5\right)}}\right)\leq\frac{\varepsilon}{5}
\]
\[
\Rightarrow J\geq\left[\left(\frac{5\,C}{\varepsilon}\right)\left(1+\sqrt{2\,\log\frac{5}{\delta}}\right)\right]^{2}.
\]
Now consider the inequalities: 
\begin{align}
\|v'-T\,v\|_{1,\,\mu}\leq\, & \|v'-T\,v\|_{1,\,\hat{\mu}}+\varepsilon/5\label{eq:L1_Bellman}\\
\leq\, & \|v'-\tilde{v}\|_{1,\,\hat{\mu}}+2\,\varepsilon/5\label{eq:L1_Bellman-1}\\
\leq\, & \|\hat{f}-\tilde{v}\|_{1,\,\hat{\mu}}+2\,\varepsilon/5\label{eq:L1_Bellman-2}\\
\leq\, & \|\hat{f}-T\,v\|_{1,\,\hat{\mu}}+3\,\varepsilon/5\label{eq:L1_Bellman-3}\\
\leq\, & \|\hat{f}-T\,v\|_{1,\,\mu}+4\,\varepsilon/5\label{eq:L1_Bellman-4}\\
\leq\, & \|f^{*}-T\,v\|_{1,\,\mu}+\varepsilon\label{eq:L1_Bellman-5}\\
\leq\, & d_{1,\,\mu}\left(T\,v,\,\mathcal{F}\left(\Theta\right)\right)+\varepsilon+\varepsilon'.\label{eq:L1_Bellman-6}
\end{align}
First, note that inequality (\ref{eq:L1_Bellman-2}) is immediate
since $\|v'-\tilde{v}\|_{1,\,\hat{\mu}}\leq\|f-\tilde{v}\|_{1,\,\hat{\mu}}$
for all $f\in\widehat{\mathcal{F}}\left(\theta^{1:J}\right)$ by the
choice of $v'$ as the minimizer in Step 3 of Algorithm 1. Second,
inequalities (\ref{eq:L1_Bellman}) and (\ref{eq:L1_Bellman-4}) follow
from Lemma \ref{lem:Bellman_empirical} by choosing $N\geq1$ to satisfy

\[
8\,e\left(J+1\right)\left(\frac{2\,e\,v_{\max}}{\varepsilon/5}\right)^{J}\exp\left(\frac{-N\,\left(\varepsilon/5\right)^{2}}{128\,v_{\max}^{2}}\right)\leq\frac{\delta}{5}
\]
\[
\Rightarrow N\geq2^{7}5^{2}\bar{v}_{\max}^{2}\log\left[\frac{40\,e\left(J+1\right)}{\delta}\left(10\,e\,\bar{v}_{\max}\right)^{J}\right].
\]
Third, inequality (\ref{eq:L1_Bellman-6}) follows from the choice
of $f^{*}\in\mathcal{F}$. Finally, inequalities (\ref{eq:L1_Bellman-1})
and (\ref{eq:L1_Bellman-3}) follow from Lemma \ref{lem:Bellman_Hoeffding}
by choosing $M\geq1$ to satisfy

\[
2\,N\,|\mathbb{A}|\,\exp\left(\frac{-2\,M\,\varepsilon^{2}}{v_{\max^{2}}}\right)\leq\frac{\delta}{5}
\]
\[
\Rightarrow M\geq\left(\frac{v_{\max}^{2}}{2\,\varepsilon^{2}}\right)\log\left[\frac{10\,N\,|\mathbb{A}|}{\delta}\right].
\]
Since $\varepsilon'$ was arbitrary, the desired result then follows
by the union bound. \end{proof} Using similar steps as Theorem \ref{thm:Bellman_L1},
the next theorem bounds the error in one iteration of EVL with respect
to $\mathcal{L}_{2,\,\mu}\left(\mathbb{S}\right)$. \begin{theorem}
\label{thm:Bellman_L2} Choose $v\in\mathcal{F}\left(\mathbb{S};\,v_{\max}\right)$,
$\varepsilon>0$, and $\delta\in\left(0,\,1\right)$. Also choose
$J\geq\left[\frac{5\,C}{\varepsilon}\left(1+\sqrt{2\,\log\frac{5}{\delta}}\right)\right]^{2}$,
$N\geq2^{7}5^{2}\bar{v}_{\max}^{4}\log\left[\frac{40\,e\left(J+1\right)}{\delta}\left(10\,e\,\bar{v}_{\max}\right)^{J}\right]$,
and $M\geq\left(\frac{v_{\max}^{2}}{2\,\varepsilon^{2}}\right)\log\left[\frac{10\,N\,|\mathbb{A}|}{\delta}\right]$.
Then, for\\
 $v'=v'\left(v,\,N,\,M,\,J,\,\mu,\,\nu\right)$ we have $\|v'-T\,v\|_{2,\,\mu}\leq d_{2,\,\mu}\left(T\,v,\,\mathcal{F}\left(\Theta\right)\right)+\varepsilon$
with probability at least $1-\delta$. \end{theorem} In the next
lemma we show that we can make the bias between the regression function
$f_{M}$ and the Bellman update $T\,v$ arbitrarily small uniformly
over $s\in\mathbb{S}$ through the choice of $M\geq1$. \begin{lemma}
\label{lem:Bellman_supremum} For any $\varepsilon>0$ and $M\geq1$,
\[
\|f_{M}-T\,v\|_{\infty}\leq\gamma\left[\varepsilon+2\,|\mathbb{A}|\,\exp\left(\frac{-2\,M\,\varepsilon^{2}}{v_{\max^{2}}}\right)\left(v_{\max}-\varepsilon\right)\right].
\]
\end{lemma} \begin{proof} For any $s\in\mathbb{S}$, we compute
\begin{align*}
 & |f_{M}\left(s\right)-\left[T\,v\right]\left(s\right)|\\
\leq\, & \mathbb{E}[|\min_{a\in\mathbb{A}}\left\{ c\left(s,\,a\right)+\frac{\gamma}{M}\sum_{m=1}^{M}v\left(X_{m}^{s,\,a}\right)\right\} \\
 & -\min_{a\in\mathbb{A}}\left\{ c\left(s,\,a\right)+\gamma\,\mathbb{E}_{X\sim Q\left(\cdot\,\vert\,s,\,a\right)}\left[v\left(X\right)\right]\right\} |]\\
\leq\, & \gamma\,\mathbb{E}\left[\max_{a\in\mathbb{A}}|\frac{1}{M}\sum_{m=1}^{M}v\left(X_{m}^{s,\,a}\right)-\mathbb{E}_{X\sim Q\left(\cdot\,\vert\,s,\,a\right)}\left[v\left(X\right)\right]|\right]\\
\leq\, & \gamma\left[\varepsilon+2\,|\mathbb{A}|\,\exp\left(\frac{-2\,M\,\varepsilon^{2}}{v_{\max^{2}}}\right)\left(v_{\max}-\varepsilon\right)\right],
\end{align*}
where the second inequality follows from Fact \ref{fact:Basic} and
the third is by the Hoeffding inequality. \end{proof}

We make use of the following RKHS function fitting result for the
one step Bellman error in the supremum norm. \begin{theorem} \label{thm:Supremum_Bellman}
Fix $v\in\mathcal{F}\left(\mathbb{S};\,v_{\max}\right)$, $\varepsilon>0$,
and $\delta\in\left(0,\,1\right)$. Also choose $N\geq\left(\frac{2\,C_{K}\kappa}{\varepsilon}\right)^{6}\log\left(4/\delta\right)^{2}$
and $M\geq\frac{v_{\max}^{2}}{2\left(\varepsilon/4\right)^{2}}\log\left(\frac{4\,|\mathbb{A}|\,\gamma\left(v_{\max}-\varepsilon/4\right)}{\left(4-2\,\gamma\right)\varepsilon}\right)$,
where $C_{K}$ is a constant independent of the dimension of $\mathbb{S}$.
Then for 
\[
\hat{f}_{\lambda}\triangleq\arg\min_{f\in\mathcal{H}_{K}}\left\{ \frac{1}{N}\sum_{n=1}^{N}\left(f\left(S_{n}\right)-Y_{n}\right)^{2}+\lambda\,\|f\|_{\mathcal{H}_{K}}^{2}\right\} ,
\]
we have $\|\hat{f}_{\lambda}-T\,v\|_{\infty}\leq\varepsilon$ with
probability at least $1-\delta$. \end{theorem} \begin{proof} By
the triangle inequality, $\|\hat{f}_{\lambda}-T\,v\|_{\infty}\leq\|\hat{f}_{\lambda}-f_{M}\|_{\infty}+\|f_{M}-T\,v\|_{\infty}$.
We choose $N\geq1$ to satisfy
\[
C_{K}\kappa\left(\frac{\log\left(4/\delta\right)^{2}}{N}\right)^{1/6}\leq\frac{\varepsilon}{2}\Rightarrow N\geq\left(\frac{2\,C_{K}\kappa}{\varepsilon}\right)^{6}\log\left(4/\delta\right)^{2},
\]
so that $\|\hat{f}_{\lambda}-f_{M}\|_{\infty}\leq\varepsilon/2$ with
probability at least $1-\delta$ by \cite[Corollary 5]{smale2005shannon}
and the fact that $\|f\|_{\infty}\leq\kappa\,\|f\|_{\mathcal{H}_{K}}$.
Then, we choose $M\geq1$ to satisfy
\[
\gamma\left[\varepsilon/4+2\,|\mathbb{A}|\,\exp\left(\frac{-2\,M\,\left(\varepsilon/4\right)^{2}}{v_{\max}^{2}}\right)\left(v_{\max}-\varepsilon/4\right)\right]\leq\frac{\varepsilon}{2}
\]
\[
\Rightarrow M\geq\frac{v_{\max}^{2}}{2\left(\varepsilon/4\right)^{2}}\log\left(\frac{4\,|\mathbb{A}|\,\gamma\left(v_{\max}-\varepsilon/4\right)}{\left(4-2\,\gamma\right)\varepsilon}\right),
\]
so that $\|f_{M}-T\,v\|_{\infty}\leq\varepsilon/2$ by Lemma \ref{lem:Bellman_supremum}.
\end{proof}

\end{document}